\documentclass[11pt,a4paper]{amsart}

\usepackage[margin=1.02in]{geometry}
\usepackage{amsmath,amssymb,amsfonts,amsthm,mathtools,mathrsfs}
\usepackage{booktabs,longtable,tabularx}
\usepackage{enumitem}
\usepackage{xcolor}
\usepackage{microtype}
\usepackage{aliascnt}
\usepackage[colorlinks=true,
  linkcolor=blue!55!black,
  citecolor=blue!55!black,
  urlcolor=blue!55!black]{hyperref}
\usepackage[nameinlink,capitalize]{cleveref}

\setlist[itemize]{leftmargin=2.1em}
\setlist[enumerate]{leftmargin=2.3em}
\newcolumntype{Y}{>{\raggedright\arraybackslash}X}
\newcolumntype{P}[1]{>{\raggedright\arraybackslash}p{#1}}

\newtheorem{theorem}{Theorem}[section]
\newaliascnt{proposition}{theorem}
\newtheorem{proposition}[proposition]{Proposition}
\aliascntresetthe{proposition}
\newaliascnt{corollary}{theorem}
\newtheorem{corollary}[corollary]{Corollary}
\aliascntresetthe{corollary}
\newaliascnt{lemma}{theorem}
\newtheorem{lemma}[lemma]{Lemma}
\aliascntresetthe{lemma}
\newaliascnt{conjecture}{theorem}

\aliascntresetthe{conjecture}
\theoremstyle{definition}
\newaliascnt{definition}{theorem}
\newtheorem{definition}[definition]{Definition}
\aliascntresetthe{definition}
\newaliascnt{example}{theorem}

\aliascntresetthe{example}
\theoremstyle{remark}
\newaliascnt{remark}{theorem}
\newtheorem{remark}[remark]{Remark}
\aliascntresetthe{remark}
\newaliascnt{warning}{theorem}

\aliascntresetthe{warning}

\crefname{theorem}{Theorem}{Theorems}
\Crefname{theorem}{Theorem}{Theorems}
\crefname{proposition}{Proposition}{Propositions}
\Crefname{proposition}{Proposition}{Propositions}
\crefname{corollary}{Corollary}{Corollaries}
\Crefname{corollary}{Corollary}{Corollaries}
\crefname{lemma}{Lemma}{Lemmas}
\Crefname{lemma}{Lemma}{Lemmas}
\crefname{conjecture}{Conjecture}{Conjectures}
\Crefname{conjecture}{Conjecture}{Conjectures}
\crefname{definition}{Definition}{Definitions}
\Crefname{definition}{Definition}{Definitions}
\crefname{example}{Example}{Examples}
\Crefname{example}{Example}{Examples}
\crefname{remark}{Remark}{Remarks}
\Crefname{remark}{Remark}{Remarks}
\crefname{warning}{Warning}{Warnings}
\Crefname{warning}{Warning}{Warnings}
\crefname{assumption}{Assumption}{Assumptions}
\Crefname{assumption}{Assumption}{Assumptions}

\DeclareMathOperator{\Span}{Span}

\DeclareMathOperator{\Conv}{Conv}

\DeclareMathOperator{\relint}{relint}
\DeclareMathOperator{\age}{age}

\DeclareMathOperator{\rank}{rank}

\DeclareMathOperator{\Gr}{Gr}

\DeclareMathOperator{\Supp}{Supp}
\DeclareMathOperator{\Spec}{Spec}

\newcommand{\C}{\mathbb C}
\newcommand{\Q}{\mathbb Q}
\newcommand{\Z}{\mathbb Z}
\newcommand{\R}{\mathbb R}
\newcommand{\A}{\mathbb A}
\newcommand{\Gm}{\mathbb G_m}

\newcommand{\IH}{\mathrm{IH}}

\newcommand{\bSigma}{\boldsymbol\Sigma}
\newcommand{\cbSigma}{\check{\boldsymbol\Sigma}}
\newcommand{\Boxop}{\operatorname{Box}}
\newcommand{\Forb}{\mathcal F_{\mathrm{orb}}}
\newcommand{\Fst}{\mathcal F_{\mathrm{st}}}

\newcommand{\Lef}{\mathbb L}

\newcommand{\an}{\mathrm{an}}

\newcommand{\Ess}{\mathcal E}
\newcommand{\Hdg}{\mathrm{Hdg}}

\title[Motivic vanishing cycles at infinity]
{Motivic Vanishing Cycles at Infinity and\\
\large Toric Landau--Ginzburg Mirror Symmetry}
\author{Haoxu Wang}
\address{Morningside Center of Mathematics, Academy of Mathematics and
Systems Science, Chinese Academy of Sciences, Beijing 100190, China}
\email{krassotkinkolya@gmail.com}
\date{August 18, 2026}

\begin{document}

\begin{abstract}
Motivic nearby cycles at infinity are usually defined from a proper
extension of a regular function to $\mathbb P^1$, whereas the
compactifications selected by Newton data generally extend the function
only as a rational map.  We show that this discrepancy can be bypassed.
Let $f:U\to\mathbb A^1$ be regular on a smooth variety, and let $X$ be a
smooth proper compactification of $U$.  If $X$ admits a reduced divisor
containing $X\setminus U$ and $\operatorname{Supp}\operatorname{div}_X(f)$,
with simple normal crossings on a neighbourhood of the polar locus,
then the direct open-pair class defined from $1/f$ agrees with
Raibaut's compactification-independent motivic nearby cycle at infinity.
The proof reduces the resolution of $X\dashrightarrow\mathbb P^1$ to a local
cancellation for blow-ups of transverse zero--pole intersections.
Consequently, motivic calculations can be carried out on the natural
rational compactification itself.

For a Laurent polynomial nondegenerate at infinity, this comparison
recovers Raibaut's essential-face formula directly from the toric boundary.
Equivariant Hodge realization, middle-degree concentration, and weighted
Ehrhart theory then give complete degree- and weight-refined Newton formulas
for nearby and relative vanishing cohomology and, through the
limiting/irregular comparison, for irregular Hodge numbers.  The
weight-forgetting specialization of the same local decomposition is governed
by a subdivision-independent orbifold polynomial; this polynomial also
computes the zero-potential stringy invariant of compatible, possibly
non-simplicial, toric stacky fans.  This gives a motivic--Hodge derivation of
the corresponding stringy Fano--Landau--Ginzburg mirror identity.  In the
full-dimensional weak Fano case, the weight refinement further matches
monodromy eigenvalues and Jordan-block sizes with inertia ages and
anticanonical Lefschetz strings, up to noncanonical complex Jordan
equivalence.
\end{abstract}
\maketitle
\setcounter{tocdepth}{1}
\tableofcontents

\section{Introduction}
\label{sec:introduction}

The geometry of a regular function is not exhausted by its fibres over
finite values.  As $t$ tends to infinity, the distant fibres form a
degeneration carrying monodromy and limiting Hodge-theoretic data.
For a Laurent polynomial, its Newton polytope suggests that this
asymptotic geometry should be readable from the toric boundary
determined by its faces.
When this polytope is full-dimensional and contains the origin in its
interior---the weak Fano situation considered below---it also supports
compact toric stacky models whose orbifold cohomology carries age
gradings and anticanonical Lefschetz operators.  It is therefore natural
to ask whether the degeneration at infinity and this age--Lefschetz
structure are two realizations of the same polyhedral data.

There is, however, a basic mismatch between this toric picture and the
usual construction of motivic nearby cycles at infinity.  The latter
uses a proper extension of the function to $\mathbb P^1$, whereas the
compactification naturally selected by the Newton polytope generally
extends the Laurent polynomial only as a rational map.  Resolving its
indeterminacies produces a proper model, but obscures the orbitwise
structure from which a Newton formula should emerge.

We prove that this mismatch can be bridged directly.  If an auxiliary
divisor containing the boundary and the support of the divisor of $f$
has simple normal crossings near the polar locus, then the direct
open-pair class on a
smooth proper rational compactification agrees with Raibaut's
compactification-independent motivic nearby cycle at infinity.  The
proof is based on a local cancellation for blow-ups of transverse
zero--pole intersections.  Thus the natural toric compactification can
be used without building its resolution into the calculation.

The paper has three main outputs.  First, the rational-compactification
comparison above makes the natural toric boundary available for direct
motivic calculations.  Second, for Newton-nondegenerate Laurent
polynomials, equivariant Hodge realization yields complete degree- and
weight-refined Newton formulas.  Third, in the weak Fano case, the
weighted local decomposition is identified with the stringy orbifold
realization of the same polytope, giving a motivic--Hodge derivation of
the subdivision-independent Fano--Landau--Ginzburg mirror identity.

The motivic class organizes these results through the following sequence
of realizations and comparisons:
\[
 \begin{aligned}
 \text{motivic nearby cycle at infinity}
 &\longrightarrow \text{nearby limit Hodge data}\\
 &\longrightarrow \text{relative limit Hodge data}
 \longrightarrow \text{irregular Hodge numbers}.
 \end{aligned}
\]
The arrows summarize the comparisons used below; they are not meant as
canonical morphisms.  At the level of Grothendieck groups, this chain
concerns alternating Hodge classes.  In the nondegenerate Laurent case,
concentration on the intrinsic torus and the complementary torus factor
recover the cohomological grading.  Equivariant Hodge realization gives the
alternating Grothendieck class of the compactly supported limit mixed
Hodge structures of a distant fibre.  Limit Poincar\'e duality and the
long exact sequence of the pair, together with the ordinary cohomology
of $U$, pass from the nearby fibre to relative vanishing cohomology.
The limiting/irregular comparison identifies
the relevant Hodge-graded dimensions with irregular Hodge numbers.
Retaining the limit-weight grading gives a genuinely finer statement:
in the full-dimensional weak Fano case, inertia ages recover the
semisimple monodromy, while anticanonical Lefschetz strings recover the
Jordan blocks of the unipotent logarithm.

\subsection{Rational compactifications and the main comparison}
\label{subsec:intro-overview}

Let $f:U\to\mathbb A^1$ be a regular function on a smooth variety over
a field $k$ of characteristic zero.  Since all constructions are additive
over connected components and a component on which $f$ is identically zero
contributes nothing at infinity, we may assume that $f$ is nonzero on every
connected component.  Raibaut's motivic nearby cycle at
infinity $\mathcal S_{f,\infty}$ is defined from a proper extension of
$f$ to $\mathbb P^1$ and is independent of that extension
\cite{raibaut2010fibre,raibaut2012singularites}.  To compare it with a
natural rational model, suppose that $X$ is a smooth proper
compactification of $U$ and write
\[
 \operatorname{div}_X(f)=D_0-D_\infty.
\]
The reciprocal $1/f$ is regular on $X\setminus|D_0|$.  Applying the
open-pair nearby-cycle construction on this open set defines a direct
class $\mathcal S_{f,\infty}^{X,\mathrm{rat}}$.  Our main result is the
following comparison.

\begin{theorem}[Direct computation on a rational compactification]
\label{thm:intro-rational-comparison}
Let $X$ be a smooth proper compactification of $U$.  Suppose that
there is a reduced divisor $\Delta$ containing $X\setminus U$ and the
support of $\operatorname{div}_X(f)$ which is simple normal crossings
on a Zariski neighbourhood of the pole divisor $|D_\infty|$.  Then
\[
 \mathcal S_{f,\infty}^{X,\mathrm{rat}}
 =\mathcal S_{f,\infty}
 \quad\text{in }\mathcal M_k^{\widehat\mu}.
\]
Consequently, the direct class is independent of the smooth proper
rational compactification satisfying this local SNC hypothesis.
\end{theorem}

The essential calculation is local.  The actual complement of the open pair
determines the nearby-cycle strata, whereas the auxiliary SNC divisor serves
only as a toroidal datum for resolving the rational map.  Choose a resolution
of the open-pair boundary compatible with that divisor.  Additional auxiliary
components merely subdivide the nearby-cycle strata, and the pieces recombine
by the scissor relation.  When a
positive-order component of $1/f$ meets a negative-order component
transversally, blowing up their intersection creates two
nearby-cycle strata whose equivariant classes cancel.  A toroidal
sequence of such zero--pole blow-ups resolves
$X\dashrightarrow\mathbb P^1$; invariance of the direct open-pair class
under each step then identifies it with Raibaut's invariant on the
resolved proper-morphism compactification.  Proposition
\ref{prop:rational-comparison} and Lemma
\ref{lem:zero-pole-cancellation} give the precise statement and
calculation.

To the best of our knowledge, no such direct comparison has been isolated
for an arbitrary smooth proper rational compactification under a local SNC
condition only along the polar locus.  The Newton formulas themselves
originate in Raibaut's construction and computation of the
compactification-independent motivic Milnor fibre at infinity
\cite{raibaut2010fibre,raibaut2012singularites}, while Matsui and Takeuchi
developed closely related Newton-polyhedral descriptions of monodromy and
motivic Milnor fibres at infinity
\cite{matsui2011monodromy,MatsuiTakeuchiInfinity}.
A related compactification-independence theorem for irregular Hodge
filtrations on Deligne--Mumford stacks is proved in
\cite{WangCompactificationIndependence}.  That result concerns filtered
twisted de Rham theory, whereas the theorem above compares a direct motivic
open-pair class with Raibaut's invariant.  This comparison makes the
natural toric compactification available for motivic calculations without
incorporating a resolution into the model.

\subsection{The realization chain}
\label{subsec:intro-polynomial-hierarchy}

We now make the realization chain described above precise.
Let $(U,f)$ be a smooth complex Landau--Ginzburg model of dimension
$d$, and let $U_t=f^{-1}(t)$ be a smooth fibre for $|t|\gg0$.  As $t$
varies near infinity, the groups
$H^k(U^{\mathrm{an}},U_t^{\mathrm{an}};\mathbb Q)$ and
$H_c^j(U_t^{\mathrm{an}};\mathbb Q)$ carry limiting mixed Hodge
structures and quasi-unipotent monodromy.
We use the loop convention fixed in
\eqref{eq:affine-and-limit-monodromy-conventions}: $T_\Phi$ is induced
by a positive loop in $s=t^{-1}$, while
$T_\infty=T_\Phi^{-1}$ is the counterclockwise large-circle
monodromy in the affine $t$-plane.
We denote the resulting filtrations by $F_{\lim}$ and $W_{\lim}$ and
use the following degree-refined scalar polynomials.  All monodromy
eigenspaces below are generalized eigenspaces after extension of
scalars to $\mathbb C$:
\begin{align}
 E_{\Phi}^{\infty}(U,f;u,v,x)
 &=\sum_{\substack{0\leq\alpha<1\\k,p,\ell}}
 \dim\operatorname{Gr}_{F_{\lim}}^p
      \operatorname{Gr}_{p+\ell}^{W_{\lim}}
      H^k(U^{\mathrm{an}},U_t^{\mathrm{an}};\mathbb C)_{
      \exp(-2\pi i\alpha)}
 u^{p+\alpha}v^\ell x^k,
 \label{eq:intro-degree-refined-vanishing}\\
 E_{\Psi}^{\infty}(U,f;u,v,x)
 &=\sum_{\substack{0\leq\alpha<1\\j,p,\ell}}
 \dim\operatorname{Gr}_{F_{\lim}}^p
      \operatorname{Gr}_{p+\ell}^{W_{\lim}}
      H_c^j(U_t^{\mathrm{an}};\mathbb C)_{\exp(2\pi i\alpha)}
 u^{p-\alpha}v^\ell x^j.
 \label{eq:intro-degree-refined-nearby}
\end{align}
The subscripts $\Phi$ and $\Psi$ distinguish the relative vanishing object
from the nearby fibre.  The variable $x$ records cohomological degree.
Together with the monodromy character, the exponents of $u$ and $v$
record the Hodge index $p$ and the weight index $p+\ell$: the exponent
of $v$ is $\ell$, while that of $u$ is $p+\alpha$ for $E_\Phi^\infty$
and $p-\alpha$ for $E_\Psi^\infty$.
Character-valued lifts, needed only for equivariant Hodge realization, are
introduced in Subsection~\ref{subsec:hodge-realization-infinity}.

The Grothendieck group directly sees only the Euler specializations at
$x=-1$.  In particular, equivariant Hodge realization followed by nearby
scalarization gives
\begin{equation}
 \operatorname{scal}_{\Psi}
 \chi_{\Hdg}^{\widehat\mu}(\mathcal S_{f,\infty})
 =E_{\Psi}^{\infty}(U,f;u,v,-1).
 \label{eq:intro-motivic-hodge-realization}
\end{equation}
If $E_{\mathrm{ord}}(U;u,v)$ is the signed character-valued Hodge
polynomial of ordinary cohomology, the action on $U$ is trivial, so we
suppress its factor $\mathbf e_0$.  The long exact sequence of
$(U,U_t)$ and limit Poincar\'e duality then give
\begin{equation}
 E_{\Phi}^{\infty}(U,f;u,v,-1)
 =E_{\mathrm{ord}}(U;u,v)
 -(uv)^{d-1}E_{\Psi}^{\infty}(U,f;u^{-1},v^{-1},-1).
 \label{eq:intro-nearby-vanishing-relation}
\end{equation}
Thus the motivic Milnor fibre first determines the equivariant
Hodge--Deligne polynomial of the nearby limit structure, and the
relative identity converts it into the signed vanishing polynomial.

We next relate these polynomials to irregular Hodge theory.  Yu
introduced the irregular Hodge filtration $F_{\mathrm{irr}}$ on the
twisted de Rham cohomology
$H^k_{\mathrm{dR}}(U,f):=\mathbb H^k\!\left(U,
(\Omega_U^\bullet,\mathrm d+\mathrm df\wedge)\right)$; this theory was
subsequently developed in
\cite{Yu2014,EsnaultSabbahYu2017,SabbahYuIrregular}.  Its graded
dimensions define the irregular Hodge numbers and their generating
polynomial
\[
 f^{\lambda,\mu}(U,f)
 :=\dim_{\mathbb C}\operatorname{Gr}_{F_{\mathrm{irr}}}^{\lambda}
 H_{\mathrm{dR}}^{\lambda+\mu}(U,f),
 \qquad
 \mathcal F(U,f;X,Y)
 :=\sum_{\lambda,\mu}f^{\lambda,\mu}(U,f)X^\lambda Y^\mu,
\]
where $\lambda,\mu\in\mathbb Q$ and $\lambda+\mu\in\mathbb Z$.
Qin--Zhang \cite[Theorem~1.1.1]{QinZhangIrregular} prove that for an
arbitrary regular function on a smooth quasi-projective complex variety
and for $p\in\mathbb Z$ and $0\leq\alpha<1$,
\begin{equation}
 \dim\operatorname{Gr}_{F_{\mathrm{irr}}}^{p+\alpha}
 H^k_{\mathrm{dR}}(U,f)
 =
 \dim\operatorname{Gr}_{F_{\lim}}^p
 H^k(U^{\an},U_t^{\an};\mathbb C)_{
 \exp(-2\pi i\alpha)}.
 \label{eq:intro-qin-zhang-comparison}
\end{equation}
Equivalently,
\begin{equation}
 \mathcal F(U,f;X,Y)
 =E_{\Phi}^{\infty}(U,f;X/Y,1,Y).
 \label{eq:intro-irregular-recovery}
\end{equation}
If $H^\bullet_{\mathrm{dR}}(U,f)$ is concentrated in degree $d$, the Euler
specialization $x=-1$ has a single sign, and
\eqref{eq:intro-irregular-recovery} becomes
\begin{equation}
 \mathcal F(U,f;X,Y)
 =(-1)^dY^dE_{\Phi}^{\infty}(U,f;X/Y,1,-1).
 \label{eq:intro-concentrated-irregular-recovery}
\end{equation}
For a Laurent polynomial nondegenerate at infinity, after passage to
its intrinsic torus, the required middle-degree concentration follows
from the Newton theorem of Adolphson--Sperber
\cite[Theorem~4.13]{AdolphsonSperberTwisted}; see
Lemma~\ref{lem:laurent-middle-degree-concentration}.  Consequently,
\eqref{eq:intro-concentrated-irregular-recovery}, together with
\eqref{eq:intro-motivic-hodge-realization} and
\eqref{eq:intro-nearby-vanishing-relation}, reduces the calculation of
$\mathcal F(U,f;X,Y)$ to
$\operatorname{scal}_{\Psi}
\chi_{\Hdg}^{\widehat\mu}(\mathcal S_{f,\infty})$ and the ordinary
Hodge polynomial of $U$.

\subsection{Newton formulas and the weak Fano mirror}
\label{subsec:intro-weak-fano-compatibility}

The mirror-theoretic role of the Newton formula becomes visible through
weighted local decomposition.  On the Landau--Ginzburg side, its local
factors encode the monodromy eigenspaces and limit-weight filtration at
infinity.  On the compact toric side, the same factors encode inertia
ages and the Lefschetz decomposition of quotient-fan intersection
cohomology.  At $v=1$, their common scalar contribution is packaged by
the polynomial $P_{\mathrm{orb}}(P;t)$ appearing below; retaining the
limit-weight grading preserves the finer Lefschetz-string and
Jordan-block data.

Let $M$ and $N=M^\vee$ be dual lattices of rank $d$, let
$P\subset N_{\mathbb R}$ be a positive-dimensional lattice polytope
containing $0$, and put $r:=\dim P$.  Fix a Laurent polynomial
$f\in\mathbb C[N]$ which is nondegenerate at infinity and has Newton
polytope at infinity $P_{f,\infty}=P$.  Applying the comparison theorem
on the toric compactification first gives Raibaut's essential-face
formula for $\mathcal S_{f,\infty}$; the relative-link calculation
accounts for the cancellation of the nonessential boundary strata.

Equivariant Hodge realization and the hypersurface formulas of
Katz--Stapledon and Stapledon
\cite{KatzStapledonLocal,StapledonMonodromy} give the full degree- and
weight-refined Newton--Hodge formula of
Theorem~\ref{thm:global-vanishing-hstar-new}.  It is a sum over the
essential faces of $P$, with local terms given by weighted $h^*$-polynomials.
The long exact sequence of the pair and limit Poincar\'e duality recover
the complete nearby polynomial $E_{\Psi}^{\infty}(T_M,f;u,v,x)$ in
Corollary~\ref{cor:full-nearby-newton-formula}.  Thus both sides of the
realization chain retain the cohomological-degree and limit-weight
variables.

At $v=1$, the weighted local decomposition collects the face
contributions in the polynomial $P_{\mathrm{orb}}(P;t)$.  This
polynomial is independent of the subdivision; see
Subsection~\ref{subsec:stringy-orbifold-poincare}.  Corollary
\ref{prop:newton-stringy-orbifold-new} then gives
\begin{equation}
 E_{\Phi}^{\infty}(T_M,f;u,1,x)
 =u^rx^r(1+ux)^{d-r}P_{\mathrm{orb}}(P;u^{-1}).
 \label{eq:intro-lg-orbifold-bridge}
\end{equation}
Combining this identity with
\eqref{eq:intro-irregular-recovery} yields
\begin{equation}
 \mathcal F(T_M,f;X,Y)
 =X^r(1+X)^{d-r}P_{\mathrm{orb}}(P;Y/X).
 \label{eq:intro-newton-irregular-bridge}
\end{equation}

We compare this torus-side formula with the stringy invariant of the
companion article \cite{StringyICarr}.  Let $\bSigma$ be any
quasi-projective stacky log-$\mathbb Q$-Gorenstein fan in $N$ with
convex radial support $P_{\bSigma}=P$, and write
$\Fst(\bSigma,g;X,Y)$ for its subdivision-independent stringy
irregular Hodge invariant.  For a simplicial presentation this is the
orbifold irregular Hodge polynomial used by Harder--Lee
\cite{HarderLee2024}; the general definition and the radial-support
notation are recalled in Section~\ref{sec:weak-fano-comparison}.
Proposition~\ref{prop:stringy-orbifold-realization-new} identifies the
zero-potential toric side as
\begin{equation}
 \Fst(\bSigma,0;X,Y)
 =(1+X)^{d-r}P_{\mathrm{orb}}(P;XY).
 \label{eq:intro-zero-potential-orbifold}
\end{equation}
Since $\Fst(\{0\},f;X,Y)=\mathcal F(T_M,f;X,Y)$,
\eqref{eq:intro-newton-irregular-bridge} and
\eqref{eq:intro-zero-potential-orbifold} give
\begin{align}
 \Fst(\{0\},f;X,Y)
 &=X^r(1+X)^{d-r}P_{\mathrm{orb}}(P;Y/X)\notag\\
 &=X^d\Fst(\bSigma,0;X^{-1},Y).
 \label{eq:intro-weak-fano-mirror}
\end{align}
This is the subdivision-independent stringy Harder--Lee identity of
\cite[Corollary~7.2]{StringyICarr}; the calculation above provides its
motivic Newton realization and fixes the relevant normalizations.

When $r=d$ and $0\in\operatorname{int}(P)$, the specialization $v=1$
forgets the limit-weight filtration and hence the nilpotent part of
monodromy.  The full weighted local--support factorization retains this
information: the local weighted $l^*$-terms determine the semisimple
spectrum, while the primitive intersection cohomology of the quotient
fans determines the Lefschetz strings and hence the Jordan-block
sizes.  Proposition~\ref{prop:weak-fano-pure-sheaf-decomposition}
realizes the same local data in an orbifold-graded Lefschetz-module
decomposition on the compact side, whose graded dimensions and
eigenvalue-graded Jordan types are subdivision-independent.

Choose a projective crepant simplicial subdivision
$\boldsymbol\Gamma\to\bSigma$, let $\mathcal X_\Gamma$ be the
corresponding toric Deligne--Mumford stack, and write
\[
 Y_t:=f^{-1}(t),
 \qquad
 V_{\mathbb C}:=H^d(T_M^{\an},Y_t^{\an};\mathbb C)
 \quad (|t|\gg0).
\]
Let $T_{\infty,s}$ and $N_\infty$ be the semisimple monodromy and the
normalized logarithm of its unipotent part, and let $A_{\mathrm{age}}$
act on each inertia sector of age $a$ by the scalar $e^{-2\pi i a}$.
Combining the two local decompositions gives
Corollary~\ref{cor:weak-fano-age-monodromy}:
\[
 \left(V_{\mathbb C},T_{\infty,s},N_\infty\right)
 \cong_{\mathrm{noncan}}
 \left(
  H^*_{\mathrm{orb}}(\mathcal X_\Gamma;\mathbb C),
  A_{\mathrm{age}},
  -c_1^{\mathrm{orb}}(\mathcal X_\Gamma)\cup(-)
 \right).
\]
Thus the scalar weak Fano mirror identity is refined to an equality of
eigenvalue-graded Jordan types.  In the simplicial, vertex-supported
case this specializes to the sectorwise comparison of
\cite[Corollary~4.5]{WangNewtonHodge}; the pure-sheaf formulation
extends it to the nonsimplicial weak Fano setting considered here.
The comparison remains numerical and noncanonical over $\mathbb C$;
it is not an isomorphism of rational polarized limit mixed Hodge
structures.

\subsection{Organization}
\label{subsec:intro-organization}

Section~\ref{sec:motivic-background} fixes the motivic and Hodge
realization conventions, and Section~\ref{sec:nearby-at-infinity}
reviews nearby cycles at infinity.  Section
\ref{sec:rational-compactifications} proves the rational-compactification
comparison by zero--pole cancellation, while Section
\ref{sec:newton-face} derives the essential-face Newton formula.
Section~\ref{sec:hodge-ehrhart} computes its Hodge realization by
weighted Ehrhart theory, introduces $P_{\mathrm{orb}}$, and extracts
the numerical semisimple and unipotent monodromy data.  Finally,
Section~\ref{sec:weak-fano-comparison} identifies $P_{\mathrm{orb}}$
with the zero-potential stringy invariant and derives the toric
Fano--Landau--Ginzburg mirror formula.  In the full-dimensional weak
Fano case, it lifts this identity to the compact-side age--Lefschetz
decomposition and proves the resulting age--monodromy correspondence.

\section{Monodromic motives and equivariant Hodge realization}
\label{sec:motivic-background}
\label{sec:conventions}

\subsection{Grothendieck rings and monodromic motives}
\label{subsec:conventions-monodromic}

Let $k$ be a field of characteristic zero.  We use the relative and
equivariant Grothendieck rings of
\cite[Sections~2.1--2.2]{guibert2006iterated}; see also
\cite[Section~1.1]{raibaut2012singularites} for the equivalent
$\mathbb G_m$-equivariant conventions used for motivic Milnor fibers at
infinity.  Let $\widehat\mu=\varprojlim_m\mu_m$.  A
$\widehat\mu$-action is
understood to factor through some $\mu_m$, and it is \emph{good} if
every orbit is contained in an affine open subset.  

For a $k$-variety $S$, let
\[
 \mathcal M_S^{\widehat\mu}
 :=K_0^{\widehat\mu}(\operatorname{Var}_S)[\Lef^{-1}]
\]
be the localized monodromic Grothendieck ring.  Here
$K_0^{\widehat\mu}(\operatorname{Var}_S)$ is generated by $S$-varieties
with good $\widehat\mu$-action over the trivial action on $S$, modulo
the equivariant scissor relation
\[
 [X\to S,\sigma]
 =[Z\to S,\sigma|_Z]
  +[X\setminus Z\to S,\sigma|_{X\setminus Z}]
\]
for every closed $\widehat\mu$-invariant subvariety $Z\subset X$, and
the affine-lifting relation
\[
 [X\times\mathbb A_k^n,\sigma]
 = [X\times\mathbb A_k^n,\sigma']
\]
whenever $\sigma$ and $\sigma'$ are two affine liftings of the same
action on $X$.  Multiplication is induced by fiber product over $S$
with the diagonal action, and $\Lef$ is the class of $\mathbb A^1_S$
with trivial action.  Composition with a morphism $q:S\to S'$ defines
$q_!: \mathcal M_S^{\widehat\mu}\to
\mathcal M_{S'}^{\widehat\mu}$.  This is the direct-image convention of
\cite[Section~2.3]{guibert2006iterated} and
\cite[Section~1.1.4]{raibaut2012singularites}: it is defined simply by
composition.

\subsection{Equivariant Hodge realization and scalarization}
\label{subsec:equivariant-hodge-scalarization}

The coefficient ring for equivariant Hodge polynomials is
\[
 \Lambda_{\widehat\mu}
 =K_0\!\left(\operatorname{Rep}^{\mathrm{fin}}_\C(\widehat\mu)\right)
 \simeq\Z[\Q/\Z].
\]
We denote by $\mathbf e_\alpha$ the class of the character with
eigenvalue $\exp(2\pi i\alpha)$ and use the involution
\[
 \mathbf e_\alpha^\vee=\mathbf e_{-\alpha}.
\]
The same superscript $\vee$ denotes coefficient-wise involution on
$\Lambda_{\widehat\mu}$-valued polynomials.

Over $\C$, the equivariant Hodge--Deligne realization is the ring
homomorphism
\begin{equation}
 \chi_{\Hdg}^{\widehat\mu}:
 \mathcal M_\C^{\widehat\mu}
 \longrightarrow
 \Lambda_{\widehat\mu}[u,v,(uv)^{-1}].
 \label{eq:hodge-realization-map-convention}
\end{equation}
For a complex variety $U$ with good finite action, define the
character-valued compactly supported and ordinary-cohomology Hodge
polynomials by
\begin{align}
 E_c(U;u,v)
 &:=\sum_{\substack{j,p,\ell\\0\leq\alpha<1}}(-1)^j
 \dim\Gr_F^p\Gr_{p+\ell}^W
 H_c^j(U;\mathbb Q)_{\exp(2\pi i\alpha)}\,
 \mathbf e_\alpha u^pv^\ell,
 \label{eq:compact-equivariant-hodge-polynomial}\\
 E_{\mathrm{ord}}(U;u,v)
 &:=\sum_{\substack{j,p,\ell\\0\leq\alpha<1}}(-1)^j
 \dim\Gr_F^p\Gr_{p+\ell}^W
 H^j(U;\mathbb Q)_{\exp(2\pi i\alpha)}\,
 \mathbf e_\alpha u^pv^\ell.
 \label{eq:ordinary-equivariant-hodge-polynomial}
\end{align}
The realization is then characterized on varieties by
\begin{equation}
 \chi_{\Hdg}^{\widehat\mu}([U])=E_c(U;u,v).
 \label{eq:hodge-realization-compact-support}
\end{equation}
If $U$ is smooth of pure dimension $d$, equivariant Poincar\'e duality
gives
\begin{equation}
 E_c(U;u,v)
 =(uv)^dE_{\mathrm{ord}}(U;u^{-1},v^{-1})^\vee.
 \label{eq:ordinary-compact-hodge-duality}
\end{equation}
When the action is trivial, both polynomials are supported on the unit
$\mathbf e_0$; in scalar formulas we identify this summand with the
ordinary coefficient ring and suppress $\mathbf e_0$.  In particular,
$\chi_{\Hdg}^{\widehat\mu}(\Lef)=uv\mathbf e_0$.

Put
\[
 \mathscr R_x
 :=\Z[u^{\Q},v^{\pm1},x^{\pm1}]
 =\bigcup_{m\geq1}
   \Z[u^{\pm1/m!},v^{\pm1},x^{\pm1}],
 \qquad
 \mathscr R:=\Z[u^{\Q},v^{\pm1}].
\]
We use two additive scalarization maps.  For $0\leq\alpha<1$, they
are characterized on monomials by
\begin{align}
 \operatorname{scal}_{\Phi}
 (\mathbf e_{-\alpha}u^pv^\ell x^k)
 &=u^{p+\alpha}v^\ell x^k,
 \label{eq:vanishing-scalarization}\\
 \operatorname{scal}_{\Psi}
 (\mathbf e_{\alpha}u^pv^\ell x^k)
 &=u^{p-\alpha}v^\ell x^k.
 \label{eq:nearby-scalarization}
\end{align}
Each character has a unique presentation of the indicated form with
$0\leq\alpha<1$, so these are well-defined injective homomorphisms of
$\Z[u^{\pm1},v^{\pm1},x^{\pm1}]$-modules from the corresponding
$\Lambda_{\widehat\mu}$-valued Laurent polynomial ring to
$\mathscr R_x$.  They are not ring homomorphisms.  In particular, all
multiplicative constructions below are performed before
scalarization.

The subscripts $\Phi$ and $\Psi$ refer respectively to the relative
vanishing and nearby monodromy conventions.  The two choices reflect
the opposite monodromy conventions for the relative vanishing object
and the compactly supported nearby object; their geometric loop
orientation is fixed in
\eqref{eq:affine-and-limit-monodromy-conventions}.  They are related by
\begin{equation}
 \operatorname{scal}_{\Phi}
 \bigl(A(u^{-1},v^{-1},x^{-1})^\vee\bigr)
 =\operatorname{scal}_{\Psi}(A)
   (u^{-1},v^{-1},x^{-1}).
 \label{eq:scalarization-duality}
\end{equation}

\section{Motivic nearby cycles at infinity}
\label{sec:nearby-at-infinity}
\label{sec:motivic-nearby}

\subsection{The open-pair nearby cycle}
\label{subsec:open-pair-nearby}

We recall the resolution formula in the form used throughout the paper.
Let $Y$ be smooth, let $V\subset Y$ be a dense open subset, and let
$g:Y\to\A^1$ be regular with $g^{-1}(0)$ nowhere dense.  Choose a log
resolution of
\[
 \bigl(Y,(Y\setminus V)\cup g^{-1}(0)\bigr).
\]
Thus we have a proper birational morphism $h:\widetilde Y\to Y$ that
is an isomorphism over
$Y\setminus\bigl((Y\setminus V)\cup g^{-1}(0)\bigr)$, with
$\widetilde Y$ smooth and
\[
 h^{-1}\bigl((Y\setminus V)\cup g^{-1}(0)\bigr)_{\mathrm{red}}
\]
a simple normal crossings divisor.  Such a resolution exists in
characteristic zero.  Write $E_i$, $i\in A$, for its irreducible
components and put
\[
 m_i=\operatorname{ord}_{E_i}(g\circ h),
 \qquad
 A_+=\{i\in A:m_i>0\}.
\]
For a nonempty subset $I\subset A_+$, define
\[
 E_I^\circ
 =\bigcap_{i\in I}E_i\setminus\bigcup_{j\notin I}E_j,
 \qquad
 m_I=\gcd(m_i:i\in I).
\]
To fix the Kummer-action convention, choose a Zariski open set $W$ on
which the $E_i$, $i\in I$, have equations $y_i=0$, and write
\[
 g\circ h=u_Iv_I^{m_I},
 \qquad
 v_I=\prod_{i\in I}y_i^{m_i/m_I},
\]
with $u_I$ a unit along $E_I^\circ\cap W$.  Over this open set define
\begin{equation}
 \widetilde E_I^\circ|_W
 =\{(p,z)\in(E_I^\circ\cap W)\times\Gm:
       z^{m_I}=u_I(p)^{-1}\},
 \label{eq:local-kummer-cover}
\end{equation}
with $\xi\in\mu_{m_I}$ acting by $z\mapsto\xi z$.  If
$v_I'=a v_I$ on an overlap, then $u_I'=u_Ia^{-m_I}$ and the rule
$z'=az$ glues the local covers.  This gives the canonical finite
\'etale $\mu_{m_I}$-cover
$\widetilde E_I^\circ\to E_I^\circ$ used below; see
\cite[Section~3.3]{denef2001geometry}.

The generalized motivic nearby cycle of the open pair is
\begin{equation}
 \mathcal S_{g;V\subset Y}
 =h_{0!}\sum_{\varnothing\ne I\subset A_+}
 (1-\Lef)^{|I|-1}[\widetilde E_I^\circ]
 \in\mathcal M_{g^{-1}(0)}^{\widehat\mu}.
 \label{eq:open-pair-nearby}
\end{equation}
Here $h_0$ is the morphism from the zero divisor on the resolution to
$g^{-1}(0)$.  The class is independent of the chosen log resolution
\cite[Proposition~3.8]{guibert2006iterated}.  When $V=Y$, it is the
Denef--Loeser motivic nearby fiber \cite{denef1998motivic}.

We shall also use the proper-modification form of this independence.
If $q:Y'\to Y$ is proper, $Y'$ is smooth, and $q$ is an isomorphism
over $V$, then
\begin{equation}
 \mathcal S_{g;V\subset Y}
 =q_{0!}\mathcal S_{g\circ q;V\subset Y'}.
 \label{eq:open-pair-proper-modification}
\end{equation}
Here $q_0$ is the induced morphism between the zero fibres; this is the
proper-modification formula of
\cite[Theorem~3.9]{guibert2006iterated}.  Notice that the components
$E_i$ in \eqref{eq:open-pair-nearby} are the components of the inverse
image of the \emph{actual} closed set
$((Y\setminus V)\cup g^{-1}(0))_{\mathrm{red}}$.  An auxiliary SNC
divisor may be used to choose coordinates or to stratify the
calculation, but its additional components are not inserted into the
definition of the open strata $E_I^\circ$.

\subsection{The motivic nearby cycle at infinity}
\label{subsec:nearby-cycle-at-infinity}

Let $f:U\to\A^1_k$ be regular with $U$ smooth.  Choose an open
immersion of $U$ into a smooth proper variety $Y$ and a morphism
\[
 f_Y:Y\longrightarrow\mathbb P^1_k
\]
extending $f$.  Put
\[
 U^\times=U\setminus f^{-1}(0),
 \qquad
 \Omega_Y=Y\setminus f_Y^{-1}(0),
 \qquad
 g_Y=1/f_Y:\Omega_Y\longrightarrow\mathbb A^1.
\]
If $p_\infty:g_Y^{-1}(0)=f_Y^{-1}(\infty)\to\Spec k$ is the
structure morphism, the motivic nearby cycle at infinity associated
with this compactification is
\begin{equation}
 \mathcal S_{f,\infty}^{(Y,f_Y)}
 :=p_{\infty!}\mathcal S_{g_Y;U^\times\subset\Omega_Y}
 \in\mathcal M_k^{\widehat\mu}.
 \label{eq:proper-compactification-nearby-infinity}
\end{equation}
Raibaut's compactification-independence theorem shows that this class
does not depend on $(Y,f_Y)$ \cite{raibaut2012singularites}.  We denote
its common value by $\mathcal S_{f,\infty}$.
Its direct expression on a smooth proper rational compactification will
be constructed in Section~\ref{sec:rational-compactifications}.  Before
turning to that calculation, we record the Hodge realization of the
compactification-independent class.

\subsection{Equivariant Hodge realization at infinity}
\label{subsec:hodge-realization-infinity}

Let $(U,f)$ be as in
Subsection~\ref{subsec:intro-polynomial-hierarchy}, put $d=\dim U$,
and write $U_t=f^{-1}(t)$ for $|t|\gg0$.  The degree-refined scalar
polynomials were defined in
\eqref{eq:intro-degree-refined-vanishing} and
\eqref{eq:intro-degree-refined-nearby}.

We first fix the geometric loop convention underlying their opposite
eigenvalue signs.  Put $s=t^{-1}$ and, after shrinking a punctured
disc $\Delta^*$ about $s=0$, consider the family of pairs
\begin{equation}
 \left(U\times\Delta^*,\mathscr Y_f^\vee\right),
 \qquad
 \mathscr Y_f^\vee
 :=\{(x,s)\in U\times\Delta^*:sf(x)=1\}.
 \label{eq:degeneration-family-at-infinity}
\end{equation}
The relative cohomology of this family and the compactly supported
cohomology of its hypersurface fibres carry the limit mixed Hodge
structures used below.  Let $T_\Phi$ denote monodromy on the relative
object induced by a positive loop in the $s$-coordinate.  The same
loop is clockwise in the affine $t$-plane.  We write
\begin{equation}
 T_\infty:=T_\Phi^{-1}
 \label{eq:affine-and-limit-monodromy-conventions}
\end{equation}
for the monodromy induced by the counterclockwise large circle
$t(\theta)=R'e^{2\pi i\theta}$.  Thus the $T_\Phi$-eigenvalue in
$E_\Phi^\infty$ is written $e^{-2\pi i\alpha}$, whereas the dual
nearby convention in $E_\Psi^\infty$ is written
$e^{2\pi i\alpha}$.

Motivic Hodge realization
uses only their signed character-valued Euler lifts
\begin{align}
 \widetilde E_{\widehat\mu}^{\Phi,\infty}(U,f;u,v)
 &:={}\sum_{\substack{0\leq\alpha<1\\k,p,\ell}}
 (-1)^k
 \dim\Gr_{F_{\lim}}^p\Gr_{p+\ell}^{W_{\lim}}
 H^k(U^{\an},U_t^{\an};\mathbb Q)_{\exp(-2\pi i\alpha)}
 \mathbf e_{-\alpha}u^pv^\ell,
 \label{eq:equivariant-global-vanishing-E}\\
 \widetilde E_{\widehat\mu}^{\Psi,\infty}(U,f;u,v)
 &:={}\sum_{\substack{0\leq\alpha<1\\j,p,\ell}}
 (-1)^j
 \dim\Gr_{F_{\lim}}^p\Gr_{p+\ell}^{W_{\lim}}
 H_c^j(U_t^{\an};\mathbb Q)_{\exp(2\pi i\alpha)}
 \mathbf e_\alpha u^pv^\ell.
 \label{eq:nearby-E-definition}
\end{align}
Thus the scalar Euler specializations are, without introducing any
two-variable abbreviation,
\begin{align}
 \operatorname{scal}_{\Phi}\!\left(
 \widetilde E_{\widehat\mu}^{\Phi,\infty}(U,f;u,v)\right)
 &=E_{\Phi}^{\infty}(U,f;u,v,-1),
 \label{eq:vanishing-Euler-scalarization}\\
 \operatorname{scal}_{\Psi}\!\left(
 \widetilde E_{\widehat\mu}^{\Psi,\infty}(U,f;u,v)\right)
 &=E_{\Psi}^{\infty}(U,f;u,v,-1).
 \label{eq:nearby-Euler-scalarization}
\end{align}

For a regular function $f:U\to\A^1$ defined over $\C$, Raibaut's
realization theorem
\cite[Theorem~2.10]{raibaut2012singularites} gives
\begin{equation}
 \chi_{\Hdg}^{\widehat\mu}(\mathcal S_{f,\infty})
 =\widetilde E_{\widehat\mu}^{\Psi,\infty}(U,f;u,v).
 \label{eq:motivic-nearby-realization-new}
\end{equation}
The long exact sequence of $(U,U_t)$ and Poincar\'e duality on the
smooth $(d-1)$-fold $U_t$ give, after Euler specialization,
\begin{equation}
 \widetilde E_{\widehat\mu}^{\Phi,\infty}(U,f;u,v)
 =E_{\mathrm{ord}}(U;u,v)
 -(uv)^{d-1}
 \widetilde E_{\widehat\mu}^{\Psi,\infty}
 (U,f;u^{-1},v^{-1})^\vee.
 \label{eq:equivariant-nearby-to-vanishing-new}
\end{equation}
Applying $\operatorname{scal}_{\Phi}$ recovers
\eqref{eq:intro-nearby-vanishing-relation}.  In general this is only
an Euler-level identity.  In the nondegenerate Laurent case,
middle-degree concentration and affine vanishing allow the full
degree-refined nearby polynomial to be reconstructed; see
Corollary~\ref{cor:full-nearby-newton-formula}.

\section{Direct computation on rational compactifications}
\label{sec:rational-compactifications}

\subsection{Direct classes on rational compactifications}
\label{subsec:rational-models}

The proper-morphism definition of $\mathcal S_{f,\infty}$ is
compactification-independent.  For direct calculations, especially in
the toric setting, it is preferable not to resolve the associated
rational map at the outset.  Let $i:U\hookrightarrow X$ be a smooth
proper compactification.  With the convention of
Subsection~\ref{subsec:intro-overview}, $f$ is nonzero on every connected
component, and we view it as an element of $k(X)^\times$.
Write
\[
 \operatorname{div}_X(f)=D_0-D_\infty,
 \qquad
 \Omega_X=X\setminus|D_0|,
 \qquad
 U^\times=U\setminus f^{-1}(0).
\]
The reciprocal
\[
 g_X=1/f
\]
is regular on $\Omega_X$.  The open-pair construction of
Subsection~\ref{subsec:open-pair-nearby} gives the direct rational
class
\begin{equation}
 \mathcal S_{f,\infty}^{X,\mathrm{rat}}
 :=p_!\mathcal S_{g_X;U^\times\subset\Omega_X}
 \in\mathcal M_k^{\widehat\mu},
 \label{eq:direct-rational-class-new}
\end{equation}
where $p:g_X^{-1}(0)\to\Spec k$ is the structure morphism.  Neither
$g_X$ nor $p$ is required to be proper.  

Because $f$ is regular on $U$, one has
\[
 U^\times=U\cap\Omega_X,
 \qquad
 g_X^{-1}(0)\subset\Omega_X\setminus U^\times.
\]
Choose a log resolution
\[
 h:\widetilde\Omega\longrightarrow\Omega_X
\]
of the open pair $U^\times\subset\Omega_X$.  Thus $h$ is an
isomorphism over $U^\times$ and
\[
 E=h^{-1}(\Omega_X\setminus U^\times)_{\mathrm{red}}
   =\bigcup_{i\in A}E_i
\]
is a simple normal crossings divisor.  Such a resolution exists in
characteristic zero.  Put
\[
 m_i=\operatorname{ord}_{E_i}(g_X\circ h),
 \qquad
 A_+=\{i\in A:m_i>0\},
\]
and, for nonempty $I\subset A_+$, set
\[
 E_I^\circ
 =\bigcap_{i\in I}E_i
  \setminus\bigcup_{j\in A\setminus I}E_j.
\]
Applying the open-pair formula of
Section~\ref{sec:nearby-at-infinity} gives
\begin{equation}
 \mathcal S_{f,\infty}^{X,\mathrm{rat}}
 =p_!h_{0!}
  \sum_{\varnothing\ne I\subset A_+}
  (1-\Lef)^{|I|-1}[\widetilde E_I^\circ].
 \label{eq:rational-snc-formula}
\end{equation}
Here $h_0$ is the map over the zero fibre and the tildes denote the
Kummer covers determined by the orders $m_i$.  The right-hand side is
independent of $h$ by
\eqref{eq:open-pair-proper-modification}.  Since motivic nearby cycles
are local along $g_X^{-1}(0)$, it is enough to choose the resolution on
a neighbourhood of that fibre.

If $(\Omega_X\setminus U^\times)_{\mathrm{red}}$ is already an SNC
divisor near $g_X^{-1}(0)$, one may take $h$ to be the identity there.
We call this the \emph{adapted SNC case}.

\subsection{Zero--pole cancellation and comparison with Raibaut}
\label{subsec:zero-pole-cancellation}

We now introduce a separate toroidal datum.  Let
\[
 \Delta=\bigcup_{\lambda\in\Lambda}\Delta_\lambda
\]
be a reduced divisor on $X$ containing $X\setminus U$ and
$\operatorname{Supp}\operatorname{div}_X(f)$.  We say that
$(X,\Delta)$ is \emph{SNC at infinity for $f$} if there is a Zariski
open neighbourhood $W$ of $|D_\infty|$ such that $\Delta|_W$ is a
simple normal crossings divisor.  No condition is imposed away from
$W$.  The divisor $\Delta$ is auxiliary: it provides toroidal
coordinates for resolving $X\dashrightarrow\mathbb P^1$, but it does
not enter either the definition \eqref{eq:direct-rational-class-new} or
the formula \eqref{eq:rational-snc-formula}.

\begin{lemma}[Zero--pole blow-up cancellation]
\label{lem:zero-pole-cancellation}
Suppose that $(X,\Delta)$ is SNC at infinity for $f$ and that two
components $\Delta_i$ and $\Delta_j$ satisfy
\[
 \operatorname{ord}_{\Delta_i}(g_X)=a>0,
 \qquad
 \operatorname{ord}_{\Delta_j}(g_X)=-b<0.
\]
Then $\Delta_i\cap\Delta_j$ is a smooth transverse intersection
contained in the SNC neighbourhood $W$.  Let
\[
 \pi:X'=\operatorname{Bl}_{\Delta_i\cap\Delta_j}X\longrightarrow X
\]
and let $\Delta'$ be the reduced total transform of $\Delta$.  Then
$(X',\Delta')$ is SNC at infinity for $f$, and the direct open-pair
classes satisfy
\begin{equation}
 \mathcal S_{f,\infty}^{X',\mathrm{rat}}
 =\mathcal S_{f,\infty}^{X,\mathrm{rat}}.
 \label{eq:zero-pole-direct-invariance}
\end{equation}
\end{lemma}

\begin{proof}
Let $E_0$ be the exceptional divisor and denote strict transforms by a
prime.  Since the center is an SNC stratum in $W$, the reduced total
transform $\Delta'$ is SNC on $\pi^{-1}(W)$.  In local coordinates
adapted to $\Delta$ one may write
\begin{equation}
 g_X=u x^ay^{-b}\prod_{s\in I}z_s^{m_s},
 \qquad
 \Delta_i=\{x=0\},\quad \Delta_j=\{y=0\},
 \label{eq:zero-pole-local-equation}
\end{equation}
where $u$ is a unit and the $z_s=0$ are the further positive-order
components of $\Delta$ through the stratum under consideration.  In
the blow-up chart $x=\beta y$, equation
\eqref{eq:zero-pole-local-equation} becomes
\begin{equation}
 g_X\circ\pi
 =u\beta^a y^{a-b}\prod_{s\in I}z_s^{m_s}.
 \label{eq:zero-pole-blowup-chart}
\end{equation}
It follows that
\[
 \operatorname{ord}_{E_0}(g_X\circ\pi)=a-b.
\]

Put $\Omega^-:=\pi^{-1}(\Omega_X)$.  The center is contained in
$|D_0|=X\setminus\Omega_X$, so
$\Omega^-\to\Omega_X$ is an isomorphism which identifies the open
subsets $U^\times$ and the functions $g_X$.  On the other hand,
$\Omega^-\subset\Omega_{X'}$, and the difference is supported on
$E_0$.  Thus it remains to compare the two open-pair classes across
the exceptional divisor.

Choose a log resolution of the open-pair boundary on $\Omega_{X'}$ which
is compatible with $\Delta'$.  Near $E_0$ this boundary is already SNC,
being a union of components of $\Delta'$, so the resolution may be chosen
to be the identity there and continued only away from that neighbourhood.
Its restriction over $\Omega^-$,
followed by the preceding isomorphism $\Omega^-\simeq\Omega_X$, is a
compatible resolution of the original open pair; after a further common
refinement away from $E_0$, if necessary, formula
\eqref{eq:open-pair-proper-modification} permits both classes to be
calculated on this single model.  Stratify the inverse image of the
center so that its incidences with both the resolved open-pair boundary
and the total transform of $\Delta'$ are constant.  Only components of
the former enter Formula~\eqref{eq:rational-snc-formula}; additional
components coming from $\Delta'$ merely subdivide its strata, and the
pieces recombine by the scissor relation.  We are therefore reduced on
each relevant piece to the monomial chart
\eqref{eq:zero-pole-local-equation}, where $I$ records the further
positive-order components of the resolved open-pair boundary.  After
pullback, the local equation retains the same monomial form; from now
on we enlarge $I$ accordingly and suppress strict transforms from the
notation.

If $a<b$, then $E_0$ belongs to the zero divisor of $f$ and is absent
from $\Omega_{X'}$, so $\Omega_{X'}=\Omega^-$ locally over the center.
If $a=b$, then $E_0$ is a component of order zero in the boundary of
$U^\times\subset\Omega_{X'}$.  It creates no indexed stratum, and its
intersections are removed from every positive-order stratum.  Hence it
again creates no new contribution.

Assume from now on that $a>b$.  Points of the center lying on a
further nonpositive-order component of the resolved boundary are
excluded from every contributing open stratum.  Fix a locally closed
piece $B$ over the remaining part of $\Delta_i\cap\Delta_j$, with $I$
the set of further positive-order boundary components through $B$, and
with all incidences with the auxiliary divisor fixed.
Over such a piece, exactly two new open strata enter
\eqref{eq:rational-snc-formula}:
\[
 B_0:=F_{I\cup\{0\}}^\circ,
 \qquad
 B_{0i}:=F_{I\cup\{0,i\}}^\circ.
\]
Here $F_K^\circ$ denotes the open stratum of the resolved open-pair
boundary indexed by $K$; the index $0$ denotes the exceptional
component and $i$ the strict transform of $\Delta_i$.  The first
stratum is the exceptional $\mathbb P^1$-bundle with the two sections
corresponding to $\Delta_i'$ and $\Delta_j'$ removed, while the second
is the section corresponding to $\Delta_i'$.  The section corresponding
to $\Delta_j'$ is outside
$\Omega_{X'}$.

Put
\begin{equation}
 c=\gcd\bigl(\{m_s:s\in I\}\cup\{a-b\}\bigr),
 \qquad
 e=\gcd(c,a).
 \label{eq:zero-pole-gcds}
\end{equation}
Let $\widetilde B_0$ and $\widetilde B_{0i}$ denote the Kummer covers
of $B_0$ and $B_{0i}$, respectively.  In the chart
\eqref{eq:zero-pole-blowup-chart}, the coordinate $\beta$ belongs to
$\Gm$ on $B_0$, and the unit remaining after the positive monomial
has been removed is $u\beta^a$.  Thus $\widetilde B_0$ is locally
\begin{equation}
 D_c
 =\{(\lambda,p,\beta)\in\Gm\times B\times\Gm:
      \lambda^c=(u(p)\beta^a)^{-1}\},
 \label{eq:zero-pole-open-kummer}
\end{equation}
where $\xi\in\mu_c$ acts by
\[
 \xi\cdot(\lambda,p,\beta)=(\xi\lambda,p,\beta).
\]
Along $B_{0i}$, the functions $\beta,y$, and $z_s$ are local
equations of the positive components with orders $a$, $a-b$, and
$m_s$, respectively.  After these positive factors are removed, the
remaining unit is $u$.  Its Kummer cover is therefore locally
\begin{equation}
 B^{(e)}
 =\{(\eta,p)\in\Gm\times B:\eta^e=u(p)^{-1}\},
 \label{eq:zero-pole-section-kummer}
\end{equation}
with its natural $\mu_e$-action.

Write $a=ea'$ and $c=ec'$.  Then
$\gcd(a',c')=1$, so choose $s,t\in\Z$ such that
\[
 sa'+tc'=1.
\]
On \eqref{eq:zero-pole-open-kummer} set
\begin{equation}
 \eta=\lambda^{c'}\beta^{a'},
 \qquad
 \theta=\lambda^{-s}\beta^t.
 \label{eq:zero-pole-monomial-change}
\end{equation}
The defining equation gives
\[
 \eta^e=\lambda^c\beta^a=u(p)^{-1}.
\]
Moreover, the exponent matrix of
\eqref{eq:zero-pole-monomial-change} has determinant
$c't+a's=1$.  Hence this is an invertible monomial change of
coordinates, with inverse
\[
 \lambda=\eta^t\theta^{-a'},
 \qquad
 \beta=\eta^s\theta^{c'}.
\]
The $\mu_c$-action becomes
\begin{equation}
 \xi\cdot(\eta,\theta)
 =\bigl(\xi^{c'}\eta,\xi^{-s}\theta\bigr).
 \label{eq:zero-pole-transformed-action}
\end{equation}
Thus the action on the first coordinate factors through
$\mu_c\to\mu_e$, $\xi\mapsto\xi^{c'}$, and the action on the
$\theta$-coordinate is linear.

The preceding monomial description is compatible with changes of
coordinates in \eqref{eq:zero-pole-local-equation}.  On overlaps, the
$\theta$-coordinates differ by invertible transition functions.
Consequently, the local products $B^{(e)}\times\Gm$ glue to the
complement of the zero section in an equivariant line bundle over
$\widetilde B_{0i}$.  The equivariant affine-bundle relation and the
scissor relation now give
\begin{equation}
 [\widetilde B_0]
 =(\Lef-1)[\widetilde B_{0i}]
 \quad\text{in }\mathcal M_k^{\widehat\mu}.
 \label{eq:zero-pole-kummer-class-relation}
\end{equation}

The index sets of $B_0$ and $B_{0i}$ have cardinalities $|I|+1$ and
$|I|+2$.  Their total contribution is therefore
\begin{align*}
 &(1-\Lef)^{|I|}[\widetilde B_0]
 +(1-\Lef)^{|I|+1}[\widetilde B_{0i}]\\
 &\qquad
 =(1-\Lef)^{|I|}
 \bigl((\Lef-1)+(1-\Lef)\bigr)
 [\widetilde B_{0i}]
 =0.
\end{align*}
Summing over the pieces $B$ and then recombining the auxiliary
subdivision proves \eqref{eq:zero-pole-direct-invariance}.  The entire
modification takes place over $W$; the reduced total transform is SNC
there, and every new positive-order component belongs to it.  Hence
$(X',\Delta')$ remains SNC at infinity.
\end{proof}

\begin{proposition}[Direct computation and compactification independence]
\label{prop:rational-comparison}
Let $X$ be a smooth proper compactification of $U$, and suppose that
there is a reduced divisor $\Delta$ containing $X\setminus U$ and
$\operatorname{Supp}\operatorname{div}_X(f)$ for which
$(X,\Delta)$ is SNC at infinity for $f$.  Then
\[
 \mathcal S_{f,\infty}^{X,\mathrm{rat}}
 =\mathcal S_{f,\infty}.
\]
Consequently the left-hand side is independent of the chosen smooth
proper rational compactification satisfying this local SNC hypothesis,
and of the auxiliary divisor $\Delta$.
\end{proposition}

\begin{proof}
We first resolve the rational map toroidally near its indeterminacy
locus.  Every indeterminacy point lies in
\[
 |D_0|\cap|D_\infty|\subset W,
\]
where the boundary is SNC.  In an \'etale coordinate chart in $W$
adapted to $\Delta$, one has
\[
 g_X=u\prod_{i=1}^s x_i^{m_i},
\]
where $u$ is a unit.  A cone of the corresponding toroidal cone complex
maps to a cone of the fan of $\mathbb P^1$ precisely when the orders on
all its rays have the same weak sign.  Suppose that two rays spanning a
mixed-sign two-cone have orders
\[
 m_i=a>0,
 \qquad
 m_j=-b<0.
\]
Blowing up $\Delta_i\cap\Delta_j$ is the star subdivision obtained by
inserting the sum of these rays, and the new exceptional divisor has
order $a-b$.  If $a>b$, the remaining mixed-sign two-cone has orders
$a-b$ and $-b$; if $a<b$, it has orders $a$ and $-(b-a)$; and if
$a=b$, the new ray has order zero and separates the two signs.
Repetition is thus the subtractive Euclidean algorithm and terminates
after finitely many steps.

Because this cone complex is simplicial, every mixed-sign cone contains
a mixed-sign two-face.  Performing the preceding continued-fraction
subdivisions compatibly on the finite part of the toroidal cone complex
meeting $|D_\infty|$ gives a sequence
\[
 \widetilde X\longrightarrow X
\]
of blow-ups of transverse intersections of opposite-order components
for which the induced rational map is everywhere regular.  The reduced
total transform of $\Delta$ remains SNC on the neighbourhood containing
the positive-order locus.  At every step,
Lemma~\ref{lem:zero-pole-cancellation} gives equality of the direct
open-pair classes; hence
\[
 \mathcal S_{f,\infty}^{X,\mathrm{rat}}
 =
 \mathcal S_{f,\infty}^{\widetilde X,\mathrm{rat}}.
\]

The resulting morphism
\[
 \widetilde f:\widetilde X\longrightarrow\mathbb P^1
\]
is proper because $\widetilde X$ is proper and $\mathbb P^1$ is
separated.  Each blow-up center is contained in a positive-order
divisor and hence is disjoint from $U$, so $\widetilde X$ is still a
compactification of $U$ and $\widetilde f$ extends $f$.  On this
resolved model, the
open-pair formula is precisely
the nearby-cycle formula at infinity associated with a proper extension
of $f$.  Raibaut's compactification-independence theorem therefore gives
\[
 \mathcal S_{f,\infty}^{\widetilde X,\mathrm{rat}}
 =\mathcal S_{f,\infty}.
\]
Combining the two equalities proves the first assertion.  Applying the
same argument to any two admissible rational compactifications proves
the compactification-independence statement.  Finally, the auxiliary
divisor $\Delta$ does not occur in the definition
\eqref{eq:direct-rational-class-new}; it is used only to construct the
toroidal sequence.  This also proves independence of its choice.
\end{proof}

\begin{remark}
The local SNC hypothesis is deliberately imposed only near
$|D_\infty|$.  Singularities of $f^{-1}(0)$ contained in $U$, and
singularities of its closure away from the polar locus, are removed from
$\Omega_X$ and do not enter the open-pair nearby cycle.  Properness is
still global: it ensures that the resolved model is a proper extension
to $\mathbb P^1$ and hence can be compared with Raibaut's invariant.
\end{remark}

\section{The Newton-face formula at infinity}
\label{sec:newton-face}

\subsection{Laurent polynomials and Newton polytopes at infinity}
\label{subsec:conventions-newton}

We now specialize the general comparison theorem to Laurent
polynomials on algebraic tori.  
Let $M$ be a lattice, let $N=M^\vee$, and write
\[
 T_M=\Spec k[N],\qquad d=\rank M.
\]
A Laurent polynomial on $T_M$ is a finite sum
\[
 f=\sum_{n\in N}a_n\chi^n\in k[N].
\]
For $f\ne0$, its support and Newton polytope at infinity are
\[
 \Supp(f)=\{n\in N:a_n\ne0\},
 \qquad
 P_{f,\infty}=\Conv\bigl(\{0\}\cup\Supp(f)\bigr)\subset N_\R.
\]
Put $r=\dim P_{f,\infty}$.  

\subsection{The toric compactification and nondegeneracy}
\label{subsec:toric-pole-faces}

For a face $F\preceq P_{f,\infty}$, write
\[
 f_F=\sum_{n\in F\cap\Supp(f)}a_n\chi^n
\]
for the corresponding face polynomial, and let
\[
 \sigma_F
 =\{m\in M_\R:\langle m,-\rangle|_{P_{f,\infty}}
       \text{ attains its minimum on }F\}
\]
be its inner normal cone.  Assume first that $r=d$.  These cones form
the complete inner normal fan $\mathcal N(P_{f,\infty})$.  Choose a
smooth complete refinement $\Sigma$ of $\mathcal N(P_{f,\infty})$.
The associated
toric variety $X_\Sigma$ is a smooth proper compactification of $T_M$.
The orbit--cone correspondence gives
\[
 X_\Sigma=\bigsqcup_{\tau\in\Sigma}O_\tau,
 \qquad
 O_\tau\simeq\Gm^{d-\dim\tau}.
\]
For a ray $\rho\in\Sigma(1)$, let $m_\rho\in M$ be its primitive
generator and let $D_\rho=\overline{O_\rho}$.  Then
\[
 X_\Sigma\setminus T_M=\bigcup_{\rho\in\Sigma(1)}D_\rho.
\]
The order of the rational function $f$ along a boundary divisor is
\begin{equation}
 \operatorname{ord}_{D_\rho}(f)
 =\min_{n\in\Supp(f)}\langle m_\rho,n\rangle.
 \label{eq:toric-pole-criterion-new}
\end{equation}
Consequently the reduced pole divisor is the union of the $D_\rho$
with $\operatorname{ord}_{D_\rho}(f)<0$; order-zero boundary divisors
remain part of the boundary but are not pole components.
We call
\[
 \left\{m\in M_\R:
 \min_{n\in\Supp(f)}\langle m,n\rangle<0\right\}
\]
the pole region.

Since $\Sigma$ refines $\mathcal N(P_{f,\infty})$, every cone
$\tau\in\Sigma$ whose relative interior lies in the pole region
has a unique supporting face $F(\tau)$ not containing $0$ such that
$\tau^\circ\subset\sigma_{F(\tau)}^\circ$.  If $\tau$ is smooth and
$\tau(1)=\{\rho_1,\ldots,\rho_q\}$, choose the corresponding toric
coordinates $y_1,\ldots,y_q$ transverse to $O_\tau$.  Locally along
$O_\tau$ one has
\begin{equation}
 f=y_1^{a_1}\cdots y_q^{a_q}g_\tau,
 \qquad
 a_i=\operatorname{ord}_{D_{\rho_i}}(f),
 \label{eq:toric-local-form-new}
\end{equation}
where $g_\tau$ is regular and its restriction to $O_\tau$ is a
monomial translate of $f_{F(\tau)}$.  Along a cone all of whose rays
are pole rays, write $a_i=-b_i$ with $b_i>0$.

We call $f$ \emph{nondegenerate at infinity} if, for every face
$F\preceq P_{f,\infty}$ not containing $0$, the hypersurface
\[
 \{f_F=0\}\subset T_M
\]
is empty or smooth of pure codimension one.  By
\eqref{eq:toric-local-form-new}, this condition implies that the
closure of $V(f)$ is smooth along the pole strata and meets the
corresponding torus orbits transversally.  Hence the toric boundary
together with the zero divisor is SNC near the zero--pole
indeterminacy locus.  It is therefore SNC on a neighbourhood of the
entire pole divisor after shrinking that neighbourhood away from the
irrelevant part of the zero divisor.  This is exactly the local SNC
hypothesis of Proposition~\ref{prop:rational-comparison}; no smoothness
of $V(f)$ away from the polar locus is required.  Consequently the
direct toric class computes $\mathcal S_{f,\infty}$.

For Proposition~\ref{prop:rational-comparison}, take the auxiliary
divisor $\Delta$ to be the union of the toric boundary and the closure
of $V(f)$.  Since $U=T_M$, one has
\[
 U^\times=T_M\setminus V(f)=\Omega_{X_\Sigma}\cap T_M,
\]
so $(\Omega_{X_\Sigma}\setminus U^\times)_{\mathrm{red}}$ is precisely
the restriction of the toric boundary.  It is SNC near
$g_{X_\Sigma}^{-1}(0)$, and hence formula
\eqref{eq:rational-snc-formula} applies with the identity as resolution.

\subsection{Essential faces and Kummer face hypersurfaces}
\label{subsec:essential-kummer-faces}

Let $\mathcal F_\infty(P_{f,\infty})$ be the set of nonempty faces of
$P_{f,\infty}$ not containing $0$.  A face
$F\in\mathcal F_\infty(P_{f,\infty})$ is \emph{essential at infinity}
if no proper face of $P_{f,\infty}$ contains both $F$ and $0$.  We write
\begin{equation}
 \Ess(P_{f,\infty})
 =\{F\in\mathcal F_\infty(P_{f,\infty}):
        F\text{ is essential at infinity}\}.
 \label{eq:essential-face-set}
\end{equation}

\begin{lemma}[Relative pole-link calculation]
\label{lem:essential-link-euler}
Assume that $r=d$, retain the smooth refinement $\Sigma$ chosen above,
and let $\Sigma^{\mathrm{pole}}$ be the set of nonzero cones
$\tau\in\Sigma$ such that $\operatorname{ord}_{D_\rho}(f)<0$ for every
ray $\rho\preceq\tau$.  For
$F\in\mathcal F_\infty(P_{f,\infty})$, set
\[
 c_F=
 \sum_{\substack{\tau\in\Sigma^{\mathrm{pole}}\\
                  \tau^\circ\subset\sigma_F^\circ}}
 (-1)^{\dim\tau-1}.
\]
Then
\begin{equation}
 c_F=
 \begin{cases}
  (-1)^{\operatorname{codim}_{P_{f,\infty}}F-1},
    &F\in\Ess(P_{f,\infty}),\\
  0,&F\notin\Ess(P_{f,\infty}).
 \end{cases}
 \label{eq:essential-link-euler}
\end{equation}
In particular, if $0\in\relint(P_{f,\infty})$, every proper face is
essential and the first case applies throughout
$\mathcal F_\infty(P_{f,\infty})$.
\end{lemma}

\begin{proof}
Intersecting the refinement of $\sigma_F$ with a sphere identifies
$c_F$ with the compactly supported Euler characteristic of the
relative pole link of $F$, with its order-zero boundary removed.  If
$F$ is essential, this link is an open ball of dimension
$\operatorname{codim}_{P_{f,\infty}}F-1$, giving the first value in
\eqref{eq:essential-link-euler}.  If a proper face contains both $F$
and $0$, the corresponding relative link is conical and its
alternating contribution vanishes.  This is the toric-link form of
Raibaut's Euler calculation
\cite[Lemma~3.20]{raibaut2012singularites}.
\end{proof}

For $F\in\mathcal F_\infty(P_{f,\infty})$, put
\[
 \begin{aligned}
  \widehat F&=\Conv(\{0\}\cup F),
  &N_F&=N\cap\Span_\R(F-F),\\
  N_{\widehat F}&=N\cap\Span_\R(\widehat F),
  &T_{\widehat F}&=\Spec k[N_{\widehat F}].
 \end{aligned}
\]
The intrinsic Kummer face hypersurface is
\begin{equation}
 V_F(f)=\{x\in T_{\widehat F}:f_F(x)=1\}.
 \label{eq:face-kummer-hypersurface}
\end{equation}
The quotient $N_{\widehat F}/N_F$ has rank one.  Its primitive grading
is constant, say of value $a_F>0$, on $F$ and induces a canonical
$\mu_{a_F}$-action on $V_F(f)$.  Indeed, for
$\zeta\in\mu_{a_F}$ the grading action sends every monomial of $f_F$ to
$\zeta^{a_F}$ times itself, and hence fixes the equation $f_F=1$.

\subsection{The motivic Newton formula}
\label{subsec:motivic-newton-formula}

\begin{theorem}[Newton formula at infinity]
\label{thm:motivic-newton-full}
Let $f\in k[N]$ be nondegenerate at infinity over a field of
characteristic zero, and put $r=\dim P_{f,\infty}$.  Then
\begin{equation}
 \mathcal S_{f,\infty}
 =(-1)^{d-r}\sum_{F\in\Ess(P_{f,\infty})}
      (1-\Lef)^{d-\dim F-1}[V_F(f)]
 \quad\text{in }\mathcal M_k^{\widehat\mu}.
 \label{eq:newton-intrinsic-face}
\end{equation}
In particular, the overall sign is positive when $P_{f,\infty}$ is
full-dimensional.
\end{theorem}

\begin{proof}
Assume first that $r=d$, and use the notation
$\Sigma^{\mathrm{pole}}$ and $c_F$ of
Lemma~\ref{lem:essential-link-euler}.  For
$\tau\in\Sigma^{\mathrm{pole}}$ of dimension $q$, write $a_i=-b_i$ in
\eqref{eq:toric-local-form-new} and put
\[
 e_\tau=\gcd(b_1,\ldots,b_q),
 \qquad
 O_\tau^\circ=O_\tau\setminus\overline{V(f)}.
\]
Let $\widetilde O_\tau^\circ\to O_\tau^\circ$ be the
$\mu_{e_\tau}$-Kummer cover occurring in the adapted SNC formula for
$1/f$.
Indeed, \eqref{eq:toric-local-form-new} gives
\[
 \frac1f
 =y_1^{b_1}\cdots y_q^{b_q}g_\tau^{-1},
\]
so the convention \eqref{eq:local-kummer-cover} identifies this cover
locally with
\[
 \{(p,z)\in O_\tau^\circ\times\Gm:
       z^{e_\tau}=g_\tau(p)\},
 \qquad \xi\cdot z=\xi z.
\]
Proposition~\ref{prop:rational-comparison} and
\eqref{eq:rational-snc-formula} give
\begin{equation}
 \mathcal S_{f,\infty}
 =\sum_{\tau\in\Sigma^{\mathrm{pole}}}
 (1-\Lef)^{\dim\tau-1}[\widetilde O_\tau^\circ].
 \label{eq:toric-snc-sum}
\end{equation}

We first identify the Kummer cover in each orbit contribution.  For a
face $F$ not containing $0$, introduce temporarily
\[
 V_F^{\mathrm{amb}}(f)=\{x\in T_M:f_F(x)=1\}.
\]
Completing the primitive vector
$(b_1/e_\tau,\ldots,b_q/e_\tau)$ to a unimodular basis and applying the
resulting monomial change of coordinates to
\eqref{eq:toric-local-form-new} gives an equivariant isomorphism
\begin{equation}
 V_{F(\tau)}^{\mathrm{amb}}(f)
 \simeq\Gm^{q-1}\times\widetilde O_\tau^\circ,
 \qquad
 [V_{F(\tau)}^{\mathrm{amb}}(f)]
 =(\Lef-1)^{q-1}[\widetilde O_\tau^\circ].
 \label{eq:kummer-orbit-face}
\end{equation}
The finite group acts on the Kummer coordinate and fixes the split
torus factors.  Splitting the character lattice transverse to
$N_{\widehat F}$ also gives
\begin{equation}
 [V_F^{\mathrm{amb}}(f)]
 =(\Lef-1)^{d-\dim F-1}[V_F(f)].
 \label{eq:ambient-intrinsic-face}
\end{equation}

Substituting \eqref{eq:kummer-orbit-face} into
\eqref{eq:toric-snc-sum} and grouping by supporting face gives
\[
 \mathcal S_{f,\infty}
 =\sum_{F\in\mathcal F_\infty(P_{f,\infty})}c_F
      [V_F^{\mathrm{amb}}(f)].
\]
Lemma~\ref{lem:essential-link-euler} and
\eqref{eq:ambient-intrinsic-face} now yield
\[
 \mathcal S_{f,\infty}
 =\sum_{F\in\Ess(P_{f,\infty})}
   (1-\Lef)^{d-\dim F-1}[V_F(f)],
\]
which proves the full-dimensional case.

For general $r$, set
\[
 N_f=N\cap\Span_\R(P_{f,\infty})
\]
and choose a splitting $N\simeq N_f\oplus N_0$.  If
$M_f=N_f^\vee$ and $T_0=\Spec k[N_0]$, then
\[
 T_M\simeq T_{M_f}\times T_0,
\]
and $f$ is pulled back from $T_{M_f}$.  The adapted SNC formula gives
\[
 \mathcal S_{f,\infty}
 =(\Lef-1)^{d-r}\mathcal S_{f|T_{M_f},\infty}.
\]
Applying the full-dimensional formula in rank $r$ gives
\eqref{eq:newton-intrinsic-face}.
\end{proof}

The identity in Theorem~\ref{thm:motivic-newton-full} is Raibaut's
Newton-polyhedral formula \cite[Theorem~3.3]{raibaut2012singularites},
but the proof given here follows a different geometric route from his.
It proceeds directly on the toric rational compactification, decomposes
the open-pair SNC expression orbit by orbit, identifies the resulting
Kummer covers with the intrinsic face hypersurfaces, and computes the
remaining coefficients by relative pole links.  This strategy is
closer in spirit to the toric calculations of Matsui--Takeuchi for
monodromy and motivic Milnor fibres at infinity
\cite{matsui2011monodromy,MatsuiTakeuchiInfinity}.  The argument above
extends that orbitwise approach to the equivariant motivic identity in
\eqref{eq:newton-intrinsic-face}: it includes the cancellation of
nonessential faces, retains the order-zero boundary, and requires
neither convenience nor full-dimensionality of $P_{f,\infty}$.

\section{Hodge realization and the orbifold specialization}
\label{sec:hodge-ehrhart}

\subsection{Weighted Ehrhart polynomials and local decompositions}
\label{subsec:weighted-ehrhart-basic}
\label{subsec:stringy-orbifold-poincare}

Fix a positive-dimensional lattice polytope $P\subset N_\R$
containing $0$, and put
\[
 N_P=N\cap\Span_\R P,
 \qquad
 r=\dim P.
\]
We work throughout this subsection with the canonical radial face fan
\begin{equation}
 \Sigma_P
 =\{c_F=\R_{\geq0}F:
       \varnothing\ne F\preceq P,\ 0\notin F\}
   \cup\{\{0\}\}.
 \label{eq:canonical-radial-face-fan}
\end{equation}
We write $\boldsymbol\Sigma_P$ for this fan endowed with the stacky
ray vectors given by the nonzero vertices of $P$.  Let $\psi_P$
be the radial height function whose unit sublevel set is $P$, and write
$\psi_c=\psi_P|_c$ for $c\in\Sigma_P$.  For a nonempty face $F$ not
containing $0$, put
\[
 \widehat F=\Conv(\{0\}\cup F),
 \qquad
 \nu_{\widehat F}=\psi_{c_F}|_{\widehat F}.
\]
Thus $\nu_{\widehat F}$ is the unique rational linear function which
vanishes at $0$ and equals $1$ on $F$.

For these pairs $(\widehat F,\nu_{\widehat F})$, we use the mixed
weighted Ehrhart polynomial
\[
 h^*(\widehat F,\nu_{\widehat F};u,v)
 \in\Lambda_{\widehat\mu}[u,v]
\]
of \cite[Section~4]{StapledonMonodromy}, normalized by specialization
at $v=1$.  Define
\begin{equation}
 h^*_{\Psi}(\widehat F,\nu_{\widehat F};u,v)
 :=\operatorname{scal}_{\Psi}\!\left(
 h^*(\widehat F,\nu_{\widehat F};u,v)\right).
 \label{eq:nearby-scalarized-hstar}
\end{equation}
Thus $\mathbf e_\alpha u^pv^\ell$ becomes
$u^{p-\alpha}v^\ell$, and scalarization is applied only after the
character-valued polynomial has been formed.

For $c\in\Sigma_P$ and $\tau\preceq c$, let $g([\tau,c];t)$ be the Stanley
$g$-polynomial of the corresponding Eulerian interval, and set
\begin{align}
 h^*_{\psi_c}(c;t)
 &:=(1-t)^{\dim c}\sum_{m\in c\cap N_P}t^{\psi_c(m)},
 \label{eq:stacky-weighted-hstar-new}\\
 l^*_{\psi_c}(c;t)
 &:=\sum_{\tau\preceq c}
 h^*_{\psi_\tau}(\tau;t)(-1)^{\dim c-\dim\tau}
 g([\tau,c]^*;t).
 \label{eq:stacky-weighted-lstar-new}
\end{align}
For the zero cone both polynomials are one.  They are respectively the
full weighted semigroup contribution and its primitive local part
\cite{StapledonWeighted,KatzStapledonLocal}.  Write
\begin{equation}
 l^*_{\psi_c}(c;t)=\sum_{a\in\Q}n_{c,a}t^a,
 \qquad
 m_c:=r-\dim c.
 \label{eq:local-age-multiplicities}
\end{equation}
The integers $n_{c,a}$ are nonnegative.  On a simplicial
presentation they count open Box elements of age $a$; intrinsically
they are the graded dimensions of the primitive multiplicity spaces
in the weighted local decomposition.  Stanley $g$-kernel inversion
gives
\begin{equation}
 h^*_{\psi_c}(c;t)
 =\sum_{\tau\preceq c}l^*_{\psi_\tau}(\tau;t)g([\tau,c];t).
 \label{eq:weighted-local-decomposition-new}
\end{equation}

The truncation $c_F\cap\{\psi_P\leq1\}$ is $\widehat F$.  Comparing the defining
weighted lattice-point series therefore gives
\begin{equation}
 h^*_{\Psi}(\widehat F,\nu_{\widehat F};t,1)
 =h^*_{\psi_{c_F}}(c_F;t).
 \label{eq:face-pyramid-cone-hstar}
\end{equation}
Indeed, the nearby-scalarized weighted Ehrhart series on the left is
$(1-t)^{-1}\sum_{m\in c_F\cap N_P}t^{\psi_{c_F}(m)}$, so the defining
denominators give the same numerator.

For $c\in\Sigma_P$, define the compactly supported and ordinary
intersection-cohomology polynomials of the quotient fan by
\begin{align}
 h_{\mathrm c}(\Sigma_P/c;t)
 &:=\sum_{c\preceq\tau\in\Sigma_P}
 (t-1)^{r-\dim\tau}g([c,\tau];t),
 \label{eq:quotient-fan-ih-compact-new}\\
 h^{\mathrm{ord}}(\Sigma_P/c;t)
 &:=t^{r-\dim c}h_{\mathrm c}(\Sigma_P/c;t^{-1}).
 \label{eq:quotient-fan-ih-new}
\end{align}
For a complete $\Sigma_P$ these agree by intersection-cohomology
Poincar\'e duality \cite{barthel2002combinatorial}.  The local factor
$l^*_{\psi_c}(c;t)$ records the primitive weighted contribution on
$c$, whereas $h^{\mathrm{ord}}(\Sigma_P/c;t)$ records its ordinary
quotient-fan support.

When $0\in\relint(P)$, the fan $\Sigma_P$ and every quotient fan
$\Sigma_P/c$ are complete.  In this case put
\[
 \IH_c^{2\bullet}:=\IH^{2\bullet}(\Sigma_P/c;\mathbb C).
\]
The height function induces a strictly convex class on the quotient
fan and hence a Lefschetz operator
\[
 L_c:\IH_c^{2j}\longrightarrow\IH_c^{2j+2}.
\]
Combinatorial hard Lefschetz gives
\begin{equation}
 \IH_c^{2\bullet}
 =\bigoplus_{0\leq j\leq\lfloor m_c/2\rfloor}
   \bigoplus_{s=0}^{m_c-2j}L_c^s\mathsf P_c^{2j},
 \label{eq:quotient-fan-primitive-decomposition}
\end{equation}
where
\begin{equation}
 \mathsf P_c^{2j}
 :=\ker\!\left(
 L_c^{m_c-2j+1}:\IH_c^{2j}
 \longrightarrow\IH_c^{2(m_c-j+1)}
 \right),
 \qquad b_{c,j}:=\dim\mathsf P_c^{2j}.
 \label{eq:quotient-fan-primitive-part}
\end{equation}
Equivalently, taking graded dimensions in
\eqref{eq:quotient-fan-primitive-decomposition} gives
\begin{equation}
 h^{\mathrm{ord}}(\Sigma_P/c;t)
 =\sum_{0\leq j\leq\lfloor m_c/2\rfloor}
 b_{c,j}t^j(1+t+\cdots+t^{m_c-2j}).
 \label{eq:weak-fano-support-lefschetz-polynomial}
\end{equation}
Thus the quotient-fan support factor is a sum of Lefschetz strings;
the string generated by a class in $\mathsf P_c^{2j}$ begins in
degree $j$ and has length $m_c-2j+1$.

It is useful to distinguish the one-variable inversion
\eqref{eq:weighted-local-decomposition-new} from the local
factorization of the limit mixed $h^*$-polynomial.  If $\mathcal S$
is a lattice polyhedral subdivision of a lattice polytope $Q$, the
latter has the form
\begin{equation}
 h^*(Q,\mathcal S;u,v)
 =\sum_{G\in\mathcal S}
 v^{\dim G+1}l^*(G;uv^{-1})
 h(\operatorname{lk}_{\mathcal S}(G);uv).
 \label{eq:limit-mixed-hstar-local-factorization}
\end{equation}
Here the local $l^*$-factor is the primitive lattice contribution,
whereas the link $h$-polynomial is the pure support contribution
\cite[Definition~8.2]{KatzStapledonLocal}.  The weighted equivariant
version retains the same two factors, with the first one carrying the
finite character \cite[Corollary~6.13]{StapledonMonodromy}.  Thus
\eqref{eq:limit-mixed-hstar-local-factorization} is not obtained from
\eqref{eq:weighted-local-decomposition-new} merely by replacing $t$
with two variables: it separates the weighted local grading, entering
through $uv^{-1}$, from the pure Hodge--Tate support grading, entering
through $uv$.

For the radial fan, collecting the face contributions turns the link
factor in \eqref{eq:limit-mixed-hstar-local-factorization} into
\[
 h^{\mathrm{ord}}(\Sigma_P/c;t),
\]
the Poincar\'e polynomial of $\IH^{2\bullet}(\Sigma_P/c)$.  Thus the
full limit-Hodge formula retains the local weighted contribution
$l^*_{\psi_c}(c;t)$ together with this quotient-fan support.  Summing
their one-variable products gives the following orbifold
specialization.

\begin{definition}
The \emph{stringy orbifold Poincar\'e polynomial} of $P$ is
\begin{equation}
 P_{\mathrm{orb}}(P;t)
 :=\sum_{c\in\Sigma_P}
 l^*_{\psi_c}(c;t)h^{\mathrm{ord}}(\Sigma_P/c;t).
 \label{eq:stringy-orbifold-poincare-new}
\end{equation}
\end{definition}

Let $\Ess(P)$ denote the faces satisfying the condition in
\eqref{eq:essential-face-set}, and augment this set by
\begin{equation}
 \widehat\Ess(P)=
 \begin{cases}
  \Ess(P)\cup\{\varnothing\},&0\in\relint(P),\\
  \Ess(P),&0\in\partial_{\mathrm{rel}}P.
 \end{cases}
 \label{eq:augmented-essential-set}
\end{equation}
Set
\[
 c_\varnothing=\{0\},
 \qquad
 \dim\varnothing=-1,
 \qquad
 \widehat\varnothing=\{0\},
 \qquad
 \nu_{\widehat\varnothing}=0,
\]
and
\[
 h^*(\{0\},0;u,v)=\mathbf e_0,
 \qquad
 h^*_{\Psi}(\{0\},0;u,v)=1.
\]
The cones $c_F$ with $F\in\widehat\Ess(P)$ are precisely the cones of
$\Sigma_P$ not contained in the relative boundary of its convex
support.

\begin{lemma}[Radial relative-link identity]
\label{lem:radial-relative-link}
For every $a\in\Sigma_P$,
\begin{equation}
 \sum_{\substack{F\in\widehat\Ess(P)\\a\preceq c_F}}
 (t-1)^{r-\dim c_F}g([a,c_F];t)
 =h^{\mathrm{ord}}(\Sigma_P/a;t).
 \label{eq:radial-relative-link}
\end{equation}
\end{lemma}

\begin{proof}
For a pure fan with convex support, the dual form of the fan
$h$-polynomial expresses ordinary intersection cohomology as the sum
of the local $g$-polynomials over the cones not contained in the
relative boundary.  Applied to $\Sigma_P/a$, this is exactly the
left-hand side of \eqref{eq:radial-relative-link}.  Equivalently, it
follows from \eqref{eq:quotient-fan-ih-compact-new} and
\eqref{eq:quotient-fan-ih-new} by Stanley $g$-kernel inversion.  When
$0\in\relint(P)$ the fan is complete and this reduces to
$h_{\mathrm c}=h^{\mathrm{ord}}$; when $0$ lies on the relative
boundary, the omitted cones form the conical boundary links.  See
\cite{barthel2002combinatorial,KatzStapledonLocal}.
\end{proof}

Combining \eqref{eq:weighted-local-decomposition-new} with
Lemma~\ref{lem:radial-relative-link}, and interchanging the two finite
sums, gives the radial local decomposition
\begin{align}
 &\sum_{F\in\widehat\Ess(P)}
 (t-1)^{r-\dim c_F}h^*_{\psi_{c_F}}(c_F;t)\notag\\
 &\qquad=
 \sum_{c\in\Sigma_P}
 l^*_{\psi_c}(c;t)
 h^{\mathrm{ord}}(\Sigma_P/c;t)\notag\\
 &\qquad=P_{\mathrm{orb}}(P;t).
 \label{eq:radial-weighted-local-decomposition}
\end{align}
This identity is the $v=1$ shadow of the local--support factorization:
it retains the total age grading but forgets the separate limit-weight
grading.

\subsection{The Newton--Hodge formula}
\label{subsec:mixed-hstar-realization}
\label{subsec:newton-hodge-formula}

Let $f\in\C[N]$ be nondegenerate at infinity, and apply the preceding
construction to $P=P_{f,\infty}$.  For
$F\in\mathcal F_\infty(P_{f,\infty})$, the function
$\nu_{\widehat F}$ in \eqref{eq:face-pyramid-cone-hstar} encodes the
canonical finite action on $V_F(f)$.  The equation $f_F-1=0$ defines a
nondegenerate hypersurface in $T_{\widehat F}$ with Newton polytope
$\widehat F$.  Hence the
equivariant Danilov--Khovanskii formula
\cite[Theorem~5.11]{StapledonMonodromy} gives
\begin{equation}
 uv\,\chi_{\Hdg}^{\widehat\mu}([V_F(f)])
 =(uv-1)^{\dim F+1}\mathbf e_0
 +(-1)^{\dim F+2}
 h^*(\widehat F,\nu_{\widehat F};u,v).
 \label{eq:face-hypersurface-hstar}
\end{equation}

We now apply this facewise realization to the motivic Newton formula.
If $r=\dim P_{f,\infty}$, Weisner's identity for the augmented set
\eqref{eq:augmented-essential-set} gives
\begin{equation}
 \sum_{F\in\widehat\Ess(P_{f,\infty})}
 (-1)^{r-\dim F-1}=1.
 \label{eq:essential-weisner-new}
\end{equation}
The empty-face convention packages the universal torus term and its
subsequent cancellation.

\begin{lemma}[Laurent middle-degree concentration]
\label{lem:laurent-middle-degree-concentration}
Let $f\in\C[N]$ be nondegenerate at infinity, put
$P=P_{f,\infty}$ and $r=\dim P$, and let $M_P=N_P^\vee$.  Then
\begin{equation}
 H^k_{\mathrm{dR}}(T_{M_P},f)=0,
 \qquad
 H^k(T_{M_P}^{\an},(T_{M_P})_t^{\an};\mathbb Q)=0
 \qquad(k\ne r).
 \label{eq:intrinsic-laurent-concentration}
\end{equation}
\end{lemma}

\begin{proof}
Choose a basis $e_1,\ldots,e_r$ of $N_P$, with corresponding
coordinates $x_1,\ldots,x_r$ on $T_{M_P}$.  For every face
$F\preceq P$ not containing the origin, choose a rational linear form
$\ell_F$ whose value on $F$ is a nonzero constant $c_F$.  The weighted
Euler identity
\[
 \sum_{i=1}^r \ell_F(e_i)x_i\frac{\partial f_F}{\partial x_i}
 =c_Ff_F
\]
shows that the logarithmic derivatives of $f_F$ have a common zero on
$T_{M_P}$ if and only if the hypersurface $\{f_F=0\}$ is singular
(with the empty case allowed).  Thus nondegeneracy at infinity is
precisely nondegeneracy relative to $P$ in the sense of
Adolphson--Sperber.  Since $P$ is full-dimensional in
$N_{P,\mathbb R}$, their Laurent concentration theorem
\cite[Theorem~4.13]{AdolphsonSperberTwisted}, specialized to their
parameter $\alpha=0$, gives the first vanishing in
\eqref{eq:intrinsic-laurent-concentration}.  The comparison
\eqref{eq:intro-qin-zhang-comparison} then gives the second.
\end{proof}

\begin{theorem}[Full degree-refined vanishing-cycle Newton formula]
\label{thm:global-vanishing-hstar-new}
Let $d=\rank N$, let $f\in\C[N]$ be nondegenerate at infinity, and put
$P=P_{f,\infty}$ and $r=\dim P$.  Then
\begin{align}
 E_{\Phi}^{\infty}(T_M,f;u,v,x)
 ={}&x^r(1+uvx)^{d-r}
 \sum_{F\in\widehat\Ess(P)}
 (uv)^{\dim F+1}(1-uv)^{r-\dim F-1}\notag\\
 &\hspace{2.1cm}\cdot
 h^*_{\Psi}
 (\widehat F,\nu_{\widehat F};u^{-1},v^{-1}).
 \label{eq:global-vanishing-hstar-new}
\end{align}
In particular,
\begin{align}
 E_{\Phi}^{\infty}(T_M,f;u,v,-1)
 ={}&(-1)^r(1-uv)^{d-r}
 \sum_{F\in\widehat\Ess(P)}
 (uv)^{\dim F+1}(1-uv)^{r-\dim F-1}\notag\\
 &\hspace{2.1cm}\cdot
 h^*_{\Psi}
 (\widehat F,\nu_{\widehat F};u^{-1},v^{-1}).
 \label{eq:global-vanishing-euler-new}
\end{align}
\end{theorem}

\begin{proof}
Apply $\chi_{\Hdg}^{\widehat\mu}$ to
\eqref{eq:newton-intrinsic-face} and use
\eqref{eq:face-hypersurface-hstar} on every essential face.  Separating
the torus terms from the weighted $h^*$-terms and using
\eqref{eq:essential-weisner-new} gives
\begin{align*}
 uv\,\widetilde E_{\widehat\mu}^{\Psi,\infty}(T_M,f;u,v)
 ={}&(uv-1)^d\mathbf e_0\\
 &-(-1)^r
 \sum_{F\in\widehat\Ess(P_{f,\infty})}
 (uv-1)^{d-\dim F-1}
 h^*(\widehat F,\nu_{\widehat F};u,v).
\end{align*}
The formal empty-face summand contributes zero before this separation,
so the augmentation does not change the motivic formula.

For the torus,
\[
 E_c(T_M;u,v)=(uv-1)^d\mathbf e_0,
 \qquad
 E_{\mathrm{ord}}(T_M;u,v)=(1-uv)^d\mathbf e_0.
\]
Substitution into the character-valued nearby/vanishing identity
\eqref{eq:equivariant-nearby-to-vanishing-new} cancels the universal
torus term.  Duality and grading reversal then give the intermediate
character-valued identity
\begin{align}
 \widetilde E_{\widehat\mu}^{\Phi,\infty}(T_M,f;u,v)
 ={}&(-1)^r
 \sum_{F\in\widehat\Ess(P_{f,\infty})}
 (uv)^{\dim F+1}(1-uv)^{d-\dim F-1}\notag\\
 &\hspace{1.7cm}\cdot
 h^*(\widehat F,\nu_{\widehat F};u^{-1},v^{-1})^\vee.
 \label{eq:equivariant-global-vanishing-hstar-new}
\end{align}
Finally, \eqref{eq:scalarization-duality} and
\eqref{eq:nearby-scalarized-hstar} give
\eqref{eq:global-vanishing-euler-new}.

It remains to restore the cohomological degree.  On the intrinsic
torus, Lemma~\ref{lem:laurent-middle-degree-concentration} shows that
the relative limit mixed Hodge structure occurs only in degree $r$.
Taking $d=r$ in \eqref{eq:global-vanishing-euler-new} therefore gives
\begin{align*}
 E_{\Phi}^{\infty}(T_{M_P},f;u,v,x)
 ={}&x^r\sum_{F\in\widehat\Ess(P)}
 (uv)^{\dim F+1}(1-uv)^{r-\dim F-1}\notag\\
 &\hspace{1.6cm}\cdot
 h^*_{\Psi}
 (\widehat F,\nu_{\widehat F};u^{-1},v^{-1}).
\end{align*}
Choose a splitting of the saturated inclusion $N_P\subseteq N$.
Then $T_M$ is the product of $T_{M_P}$ with a torus of dimension
$d-r$, and $f$ is pulled back from the intrinsic factor.  The full
ordinary Hodge--degree polynomial of the complementary torus is
$(1+uvx)^{d-r}$.  Since this factor has trivial monodromy,
degreewise K\"unneth is compatible with scalarization and gives
\eqref{eq:global-vanishing-hstar-new}.
\end{proof}

We next extract from the same formula the semisimple and unipotent
parts of monodromy.  For the operators fixed in
\eqref{eq:affine-and-limit-monodromy-conventions}, write
\begin{align}
 T_\Phi&=T_{\Phi,s}T_{\Phi,u},
 &N_\Phi&=(2\pi i)^{-1}\log T_{\Phi,u},\notag\\
 T_\infty&=T_{\infty,s}T_{\infty,u},
 &N_\infty&=(2\pi i)^{-1}\log T_{\infty,u}=-N_\Phi.
 \label{eq:normalized-monodromy-logarithms}
\end{align}

\begin{corollary}[Newton multiplicities for monodromy at infinity]
\label{cor:newton-monodromy-multiplicities}
In addition to the hypotheses of
Theorem~\ref{thm:global-vanishing-hstar-new}, suppose that $P$ is
full-dimensional and $0\in\operatorname{int}(P)$, and write
$Y_t=f^{-1}(t)$ for $|t|\gg0$.  Set
\begin{equation}
 V_{\mathbb C}:=
 H^d(T_M^{\an},Y_t^{\an};\mathbb C),
 \qquad
 V_\lambda:=\ker(T_{\infty,s}-\lambda\operatorname{id}).
 \label{eq:weak-fano-semisimple-eigenspace}
\end{equation}
Then
\begin{equation}
 \dim V_\lambda
 =\sum_{\substack{a\in\mathbb Q\\e^{-2\pi ia}=\lambda}}
   \sum_{c\in\Sigma_P}
   n_{c,a}\dim\IH_c^{2\bullet}.
 \label{eq:weak-fano-eigenvalue-multiplicity-local}
\end{equation}
Moreover, if $J_{k,\lambda}$ denotes the number of Jordan blocks of
size $k$ of $N_\infty$ on $V_\lambda$, then
\begin{equation}
 J_{k,\lambda}
 =\sum_{\substack{a\in\mathbb Q\\e^{-2\pi ia}=\lambda}}
   \sum_{\substack{c\in\Sigma_P,\;
                    0\leq j\leq\lfloor m_c/2\rfloor\\
                    m_c-2j+1=k}}
   n_{c,a}b_{c,j}.
 \label{eq:weak-fano-jordan-block-number}
\end{equation}
\end{corollary}

\begin{proof}
Fix a base point $t$ with $|t|\gg0$.  Parallel transport of the
pair $(T_M^{\an},Y_t^{\an})$ around a counterclockwise large circle
in the $t$-plane defines $T_\infty$ on $V_{\mathbb C}$.  Hence
\begin{equation}
 \det(z-T_{\infty,s})
 =\prod_\lambda(z-\lambda)^{\dim V_\lambda}.
 \label{eq:weak-fano-semisimple-characteristic-by-eigenspace}
\end{equation}
Choose the unique $\alpha\in[0,1)$ such that
$\lambda^{-1}=e^{-2\pi i\alpha}$.  Since
$T_\Phi=T_\infty^{-1}$, $V_\lambda$ is the generalized
$e^{-2\pi i\alpha}$-eigenspace for $T_\Phi$.  The definition of
the two finite filtrations, Lemma~\ref{lem:laurent-middle-degree-concentration},
Theorem~\ref{thm:global-vanishing-hstar-new}, and the weighted
equivariant form of
\eqref{eq:limit-mixed-hstar-local-factorization} give
\begin{align}
 \dim V_\lambda
 &=\left.
   \sum_{\rho\in\alpha+\mathbb Z}
   [u^\rho]E_\Phi^\infty(T_M,f;u,v,x)\,u^\rho
   \right|_{u=v=x=1}                                           \notag\\
 &=\sum_{c\in\Sigma_P}
   \left(
    \sum_{\substack{a\in\mathbb Q\\e^{-2\pi ia}=\lambda}}
    [t^a]l^*_{\psi_c}(c;t)
   \right)
   h^{\mathrm{ord}}(\Sigma_P/c;1)                              \notag\\
 &=\sum_{\substack{a\in\mathbb Q\\e^{-2\pi ia}=\lambda}}
   \sum_{c\in\Sigma_P}
   n_{c,a}\dim\IH_c^{2\bullet}.
 \label{eq:weak-fano-eigenvalue-multiplicity-from-E}
\end{align}
Here $[u^\rho]$ and $[t^a]$ denote coefficient extraction.  The
first line retains precisely the $u$-exponents with fractional part
$\alpha$, hence the $T_\Phi$-eigenvalue $\lambda^{-1}$; the second
retains precisely the local ages with $T_\infty$-eigenvalue
$\lambda$.  This proves
\eqref{eq:weak-fano-eigenvalue-multiplicity-local}.

We now use $r\geq0$ as the distance from the center of the weight
filtration; the dimension of the full-dimensional polytope is $d$.
For $a\in\mathbb Q$, define the numerical age-$a$ contribution
\begin{equation}
 w_{a,r}
 :=\sum_{\substack{c\in\Sigma_P,\ q\geq0\\m_c-2q=r}}
 n_{c,a}\dim\IH_c^{2q}.
 \label{eq:weak-fano-age-weight-contribution}
\end{equation}
Put
\[
 \nu_\lambda:=
 \begin{cases}
  d,&\lambda=1,\\
  d-1,&\lambda\ne1.
 \end{cases}
\]
Coefficient extraction from the same limit-Hodge formula gives
\begin{align}
 \dim\operatorname{Gr}_{\nu_\lambda+r}^{W_{\lim}}V_\lambda
 &=\sum_{\substack{p,\ell\in\mathbb Z\\
                    p+\ell=\nu_\lambda+r}}
   [u^{p+\alpha}v^\ell x^d]
   E_\Phi^\infty(T_M,f;u,v,x)                                  \notag\\
 &=\sum_{\substack{a\in\mathbb Q\\e^{-2\pi ia}=\lambda}}
   w_{a,r}.
 \label{eq:weak-fano-weight-dimension-from-E}
\end{align}
On $V_\lambda$, $W_{\lim}$ is the monodromy-weight filtration of
$N_\Phi$, centered at $\nu_\lambda$; see
\cite[Theorem~6.12]{StapledonMonodromy} and
\cite{sabbah1997monodromy,sabbah2018some,WangNewtonHodge}.
In particular,
\[
 N_\Phi^r:\operatorname{Gr}_{\nu_\lambda+r}^{W_{\lim}}V_\lambda
 \xrightarrow{\sim}
 \operatorname{Gr}_{\nu_\lambda-r}^{W_{\lim}}V_\lambda.
\]
For a nilpotent operator with this centered weight filtration, the
number of blocks of size $r+1$ is the dimension of the upper graded
piece at distance $r$ minus that at distance $r+2$.  By
\eqref{eq:weak-fano-weight-dimension-from-E}, this number is
\begin{equation}
 \sum_{\substack{a\in\mathbb Q\\e^{-2\pi ia}=\lambda}}
 \left(w_{a,r}-w_{a,r+2}\right).
 \label{eq:monodromy-weight-jordan-difference}
\end{equation}
For each fixed $a$, with the convention $\IH_c^{-2}=0$,
\begin{align}
 w_{a,r}-w_{a,r+2}
 &=\sum_{\substack{c\in\Sigma_P,
                    0\leq j\leq\lfloor m_c/2\rfloor\\
                    m_c-2j=r}}
 n_{c,a}\left(\dim\IH_c^{2j}-\dim\IH_c^{2j-2}\right)
                                                                    \notag\\
 &=\sum_{\substack{c\in\Sigma_P,
                    0\leq j\leq\lfloor m_c/2\rfloor\\
                    m_c-2j=r}}
 n_{c,a}b_{c,j}.
 \label{eq:weak-fano-jordan-blocks-from-weight-difference}
\end{align}
The last equality is the degree-$j$ coefficient difference in
\eqref{eq:weak-fano-support-lefschetz-polynomial}.  Substitution in
\eqref{eq:monodromy-weight-jordan-difference}, followed by
$k=r+1$, gives \eqref{eq:weak-fano-jordan-block-number}.  Finally,
$N_\infty=-N_\Phi$ by
\eqref{eq:normalized-monodromy-logarithms}, and multiplication by
$-1$ does not change Jordan-block sizes.
\end{proof}

\begin{corollary}[Full degree-refined nearby-cycle Newton formula]
\label{cor:full-nearby-newton-formula}
Retain the hypotheses of
Theorem~\ref{thm:global-vanishing-hstar-new}, and assume $r>0$.
Then
\begin{align}
 E_{\Psi}^{\infty}(T_M,f;u,v,x)
 ={}&\bigl[x(1+uvx)\bigr]^{d-r}(uvx^2)^{r-1}\notag\\
 &\quad\cdot\Bigl[
 (1+u^{-1}v^{-1}x^{-1})^r
 -(u^{-1}v^{-1}x^{-1})^r\notag\\
 &\hspace{2.4cm}
 +x^{1-r}\bigl(
 E_{\Phi}^{\infty}(T_{M_P},f;u^{-1},v^{-1},1)
 -(uv)^{-r}\bigr)\Bigr].
 \label{eq:full-nearby-newton-formula}
\end{align}
\end{corollary}

\begin{proof}
First work on $T_{M_P}$ and write $Y_t=f^{-1}(t)$ for a smooth fibre
near infinity.  By Lemma~\ref{lem:laurent-middle-degree-concentration}
and the long exact sequence of $(T_{M_P},Y_t)$,
\[
 H^j(Y_t)\simeq H^j(T_{M_P})\quad(j<r-1),
\]
while in degree $r-1$ there is an exact sequence
\[
 0\longrightarrow H^{r-1}(T_{M_P})
 \longrightarrow H^{r-1}(Y_t)
 \longrightarrow H^r(T_{M_P},Y_t)
 \longrightarrow H^r(T_{M_P})\longrightarrow0.
\]
The smooth affine $(r-1)$-fold $Y_t$ has no ordinary cohomology above
degree $r-1$.  Strictness for limit mixed Hodge structures, together
with $H^j(T_{M_P})$ of type $(j,j)$, shows that the ordinary
degree-refined polynomial of $Y_t$ in the $T_\Phi$-monodromy
convention is
\begin{equation}
 (1+uvx)^r-(uvx)^r
 +x^{r-1}\bigl(
 E_{\Phi}^{\infty}(T_{M_P},f;u,v,1)-(uv)^r\bigr).
 \label{eq:intrinsic-ordinary-fiber-expression}
\end{equation}

Limit Poincar\'e duality on $Y_t$ has dimension $r-1$, inverts the
monodromy, and transforms
\eqref{eq:intrinsic-ordinary-fiber-expression} into the bracketed
factor of \eqref{eq:full-nearby-newton-formula}, multiplied by
$(uvx^2)^{r-1}$.
Finally, the ambient fibre is the product of $Y_t$ with
$(\mathbb C^*)^{d-r}$.  The compactly supported Hodge--degree
polynomial of one copy of $\mathbb C^*$ is
$x+uvx^2=x(1+uvx)$, so compactly supported K\"unneth gives
\eqref{eq:full-nearby-newton-formula}.
\end{proof}

\begin{corollary}[Newton--orbifold specialization]
\label{prop:newton-stringy-orbifold-new}
For every $f\in\C[N]$ nondegenerate at infinity with
$P_{f,\infty}=P$,
\begin{equation}
 E_{\Phi}^{\infty}(T_M,f;u,1,x)
 =u^rx^r(1+ux)^{d-r}
 P_{\mathrm{orb}}(P;u^{-1}).
 \label{eq:weak-fano-degree-refined-new}
\end{equation}
\end{corollary}

\begin{proof}
Put $t=u^{-1}$.
For $F\in\widehat\Ess(P)$, write $c_F=\R_{\geq0}F$.
Since
\[
 u^{\dim c_F}(1-u)^{r-\dim c_F}
 =u^r(t-1)^{r-\dim c_F},
\]
equations \eqref{eq:face-pyramid-cone-hstar} and
\eqref{eq:radial-weighted-local-decomposition} give
\begin{align*}
 &\sum_{F\in\widehat\Ess(P)}
 u^{\dim F+1}(1-u)^{r-\dim F-1}
 h^*_{\Psi}
 (\widehat F,\nu_{\widehat F};u^{-1},1)\\
 &\hspace{4cm}=u^rP_{\mathrm{orb}}(P;u^{-1}).
\end{align*}
The result is now the specialization $v=1$ of
\eqref{eq:global-vanishing-hstar-new}.
\end{proof}

When $P$ is full-dimensional and $0\in\operatorname{int}(P)$,
Corollary~\ref{cor:newton-monodromy-multiplicities} extracts the
semisimple and unipotent monodromy data from the full two-variable
Newton expression.  Proposition~\ref{prop:weak-fano-pure-sheaf-decomposition}
realizes the same multiplicities in the compact-side age--Lefschetz
decomposition.

\begin{corollary}[Irregular Hodge specialization]
\label{cor:newton-irregular-hodge-new}
Under the hypotheses of
Corollary~\ref{prop:newton-stringy-orbifold-new},
\begin{equation}
 \mathcal F(T_M,f;X,Y)
 =X^r(1+X)^{d-r}
 P_{\mathrm{orb}}(P;Y/X).
 \label{eq:weak-fano-polynomial-new}
\end{equation}
\end{corollary}

\begin{proof}
By \eqref{eq:intro-irregular-recovery} and
\eqref{eq:weak-fano-degree-refined-new},
\begin{align*}
 \mathcal F(T_M,f;X,Y)
 &=E_{\Phi}^{\infty}(T_M,f;X/Y,1,Y)\\
 &=X^r(1+X)^{d-r}P_{\mathrm{orb}}(P;Y/X).
\end{align*}
\end{proof}

\section{The weak Fano mirror comparison}
\label{sec:weak-fano-comparison}

We first recall the stringy notation used in this section.  For a
smooth Deligne--Mumford Landau--Ginzburg model $(\mathcal X,g)$, write
$I\mathcal X=\coprod_{\mathfrak g}\mathcal X_{\mathfrak g}$, let
$a_{\mathfrak g}$ be the age of the component, and let
$g_{\mathfrak g}$ be the restricted potential.  Its orbifold
irregular Hodge polynomial is
\begin{equation}
 \Forb(\mathcal X,g;X,Y)
 :=\sum_{\mathfrak g}(XY)^{a_{\mathfrak g}}
 \mathcal F(\mathcal X_{\mathfrak g},g_{\mathfrak g};X,Y).
 \label{eq:orbifold-irregular-hodge-body}
\end{equation}
This is the invariant used by Harder--Lee for smooth stacky Clarke
pairs \cite{HarderLee2024}.

Let $\bSigma=(\Sigma,\{b_\rho\})$ be a quasi-projective stacky
log-$\mathbb Q$-Gorenstein fan with convex radial support in a lattice
$L$.  Its radial stacky support is
\begin{equation}
 P_{\bSigma}
 :=\bigcup_{c\in\Sigma}
 \operatorname{Conv}
 \bigl(\{0\}\cup\{b_\rho:\rho\preceq c\}\bigr).
 \label{eq:radial-stacky-support-body}
\end{equation}
For a toric Landau--Ginzburg datum $(\bSigma,g)$, choose a projective
ray-preserving crepant simplicial subdivision
$\boldsymbol\Gamma\to\bSigma$ and let $U_{\boldsymbol\Gamma}$ be the
resulting smooth toric Deligne--Mumford model.  The companion article
proves that
\begin{equation}
 \Fst(\bSigma,g;X,Y)
 :=\Forb(U_{\boldsymbol\Gamma},g;X,Y)
 \label{eq:stringy-irregular-hodge-body}
\end{equation}
is independent of the chosen subdivision
\cite[Theorem~7.1]{StringyICarr}; it therefore extends \eqref{eq:orbifold-irregular-hodge-body}
to possibly nonsimplicial stacky fans.

Building on Clarke's dual-fan construction \cite{ClarkeDualFans}, data
$(\bSigma,g)$ and $(\cbSigma,\check g)$ in dual lattices form a
stacky Clarke mirror pair
when
\begin{equation}
 P_{g,\infty}=P_{\cbSigma},
 \qquad
 P_{\check g,\infty}=P_{\bSigma}.
 \label{eq:clarke-polytope-pair-body}
\end{equation}
The stringy extension of the Harder--Lee theorem is
\begin{equation}
 \Fst(\bSigma,g;X,Y)
 =X^d\Fst(\cbSigma,\check g;X^{-1},Y)
 \label{eq:stringy-mirror-body}
\end{equation}
when the ambient lattices have rank $d$
\cite[Corollary~7.2]{StringyICarr}.  We shall recover its weak Fano
specialization from the motivic Newton calculation.

\subsection{The semi-projective stringy realization}
\label{subsec:stringy-orbifold-realization}

Let $\bSigma=(\Sigma,\{b_\rho\})$ be a quasi-projective stacky
log-$\mathbb Q$-Gorenstein fan with convex radial support in a lattice
$L$ of rank $d$, and put
\[
 P:=P_{\bSigma},
 \qquad
 L_\Sigma:=L\cap\Span_{\mathbb R}|\Sigma|,
 \qquad r:=\rank L_\Sigma=\dim P.
\]

\begin{proposition}[Stringy realization]
\label{prop:stringy-orbifold-realization-new}
With the notation above,
\begin{equation}
 \Fst(\bSigma,0;X,Y)
 =(1+X)^{d-r}P_{\mathrm{orb}}(P;XY).
 \label{eq:stringy-zero-potential-new}
\end{equation}
If $\boldsymbol\Gamma\to\bSigma$ is a projective crepant simplicial
subdivision, then
\begin{equation}
 P_{\mathrm{orb}}(P;t)
 =\sum_{q\in\mathbb Q}
 \dim H_{\mathrm{orb}}^{2q}
 (\mathcal X_{\boldsymbol\Gamma,L_\Sigma};\mathbb C)t^q.
 \label{eq:crepant-subdivision-orbifold-new}
\end{equation}
In particular, the right-hand side is independent of the chosen
crepant simplicial subdivision.
\end{proposition}

\begin{proof}
First work in the intrinsic lattice $L_\Sigma$.  Let
$\boldsymbol\Gamma\to\bSigma$ be a projective crepant simplicial
subdivision, with underlying fan $\Gamma$ and induced height function
$\psi$.  Its inertia components are indexed by open Box elements, and
their age contributions are
\[
 l^*_{\psi_c}(c;t)
 =\sum_{m\in\Boxop^\circ(c)}t^{\psi_c(m)},
\]
whereas the ordinary cohomology of the corresponding quotient fan is
$h^{\mathrm{ord}}(\Gamma/c;t)$.  Hence its sector polynomial is
\[
 \sum_{c\in\Gamma}
 l^*_{\psi_c}(c;t)h^{\mathrm{ord}}(\Gamma/c;t).
\]
The weighted local decomposition theorem shows that this sum is
unchanged under projective crepant subdivision
\cite{StapledonWeighted,KatzStapledonLocal}.  Since $\Gamma$ and the
canonical radial face fan $\Sigma_P$ admit a common projective
crepant refinement, the sum equals
$P_{\mathrm{orb}}(P;t)$ as defined in
\eqref{eq:stringy-orbifold-poincare-new}.

The Borisov--Chen--Smith formula in the complete case
\cite{borisov2005orbifold} and the semi-projective orbifold Chow
calculation of Jiang--Tseng \cite{JiangTsengSemiProjective} now give
\eqref{eq:crepant-subdivision-orbifold-new}.  For the original
possibly non-simplicial fan, the zero-potential specialization of
\cite[Theorem~7.1]{StringyICarr} identifies this common subdivision
value with its stringy invariant.

Finally, after choosing a splitting of the saturated inclusion
$L_\Sigma\subseteq L$, one has
\[
 U_{\boldsymbol\Gamma}
 \simeq(\mathbb C^*)^{d-r}
       \times\mathcal X_{\boldsymbol\Gamma,L_\Sigma}.
\]
Since
$\mathcal F((\mathbb C^*)^{d-r},0;X,Y)=(1+X)^{d-r}$,
the K\"unneth formula gives \eqref{eq:stringy-zero-potential-new}.
\end{proof}

We now lift the full-dimensional case of
\eqref{eq:crepant-subdivision-orbifold-new} from an equality of graded
dimensions to an orbifold-graded Lefschetz-module decomposition.  Assume
for the remainder of this subsection that $r=d$ and
$0\in\operatorname{int}(P)$, and write $N:=L$.  Retain from
Section~\ref{subsec:weighted-ehrhart-basic} the radial fan $\Sigma_P$,
its stacky enhancement $\boldsymbol\Sigma_P$, the local multiplicities
$n_{c,a}$ of \eqref{eq:local-age-multiplicities}, and the quotient-fan
Lefschetz data
$(\IH_c^{2\bullet},L_c,\mathsf P_c^{2j},b_{c,j})$ of
\eqref{eq:quotient-fan-primitive-decomposition}--
\eqref{eq:weak-fano-support-lefschetz-polynomial}.  Here
$m_c=d-\dim c$.

Choose a projective crepant simplicial subdivision
\[
 \boldsymbol\Gamma\longrightarrow\boldsymbol\Sigma_P
\]
and abbreviate the resulting smooth proper toric
Deligne--Mumford stack by
$\mathcal X_\Gamma:=\mathcal X_{\boldsymbol\Gamma,N}$.
Its anticanonical operator
\begin{equation}
 L:=c_1^{\mathrm{orb}}(\mathcal X_\Gamma)\cup(-)
 \label{eq:orbifold-anticanonical-operator}
\end{equation}
is represented by $\psi_P$.  It is nef and big, but need not be
ample: it is the pullback of the ample class under the toric
anticanonical contraction
\[
 \pi_\Gamma:\mathcal X_\Gamma\longrightarrow X_{\Sigma_P}.
\]
Thus ordinary hard Lefschetz for $L$ on each inertia sector cannot
simply be assumed.

\begin{proposition}[Compact-side age--Lefschetz refinement]
\label{prop:weak-fano-pure-sheaf-decomposition}
There is an $L$-equivariant noncanonical decomposition
\begin{equation}
 \bigoplus_{\mathfrak g\subset I\mathcal X_\Gamma}
 H^{2\bullet}(\mathcal X_{\mathfrak g};\mathbb C)
 \cong
 \bigoplus_{c\in\Sigma_P}\bigoplus_{a\in\mathbb Q}
 \mathbb C^{n_{c,a}}\otimes\IH_c^{2\bullet}.
 \label{eq:weak-fano-pure-sheaf-decomposition}
\end{equation}
On the left, a class in $H^{2j}(\mathcal X_{\mathfrak g};\mathbb C)$
has orbifold degree $2(j+\age(\mathfrak g))$.  On the right, the
summand $\mathbb C^{n_{c,a}}\otimes\IH_c^{2j}$ has orbifold degree
$2(a+j)$.  The decomposition preserves this total degree, and $L$
acts on the right as $1\otimes L_c$.
The local multiplicity degree $a$ need not equal the age of an
individual inertia sector: integral degree shifts may enter the
decomposition.  Its residue class modulo $\mathbb Z$, however, agrees
with the sector age modulo $\mathbb Z$ and thus determines the same
eigenvalue $e^{-2\pi ia}$.
The multiplicity spaces are canonical,
but a splitting of the displayed direct sum is not.  Forgetting the
$L$-action and taking graded dimensions recovers
\eqref{eq:crepant-subdivision-orbifold-new}.
\end{proposition}

\begin{proof}
The cohomological realization used in the proof of
Proposition~\ref{prop:stringy-orbifold-realization-new} admits a
chain-level refinement.  Under the Borisov--Chen--Smith indexing of
inertia sectors by Box elements, the Box--Koszul comparison sends a
sector label to the corresponding monomial in
$\mathcal B_{\Sigma_P}^\Gamma$, whose $\psi_P$-degree is precisely the
age of that sector.  After taking the Koszul quotient by the global
linear relations and then global sections, it therefore identifies
the zero-potential realization of $\mathcal B_{\Sigma_P}^\Gamma$ with
\[
 \bigoplus_{\mathfrak g\subset I\mathcal X_\Gamma}
 H^{2\bullet}(\mathcal X_{\mathfrak g};\mathbb C)
\]
with its total orbifold grading: a sector class of ordinary degree
$2j$ and age $\age(\mathfrak g)$ has total degree
$2(j+\age(\mathfrak g))$.  The inertia sectors are smooth proper
toric stacks, so this is their ordinary, purely even cohomology.  The
identification follows from the BCS inertia description
\cite{borisov2005orbifold} and the Box--Koszul and simplicial
comparisons
\cite[Proposition~4.1 and Theorem~4.2]{StringyICarr}.

The pure-sheaf decomposition theorem then gives a noncanonical
decomposition
\[
 \mathcal B_{\Sigma_P}^\Gamma
 \simeq\bigoplus_cR_1^\Gamma(c)\otimes
 \mathcal L_{\Sigma_P}^c,
 \qquad
 \operatorname{Hilb}R_1^\Gamma(c)=l^*_{\psi_c}(c;t),
\]
whose multiplicity spaces are the canonical primitive kernels of the
boundary maps
\cite[Theorem~5.2 and Proposition~6.4]{StringyICarr}.  Applying the
same zero-potential realization termwise, the primitive--support
decomposition identifies the contribution of the simple pure sheaf
$\mathcal L_{\Sigma_P}^c$ with
\[
 \IH^{2\bullet}(\Sigma_P/c)=\IH_c^{2\bullet}
\]
\cite[Proposition~6.6]{StringyICarr}.  On the other hand,
$\operatorname{Hilb}R_1^\Gamma(c)=l^*_{\psi_c}(c;t)$ says precisely
that the degree-$a$ part of $R_1^\Gamma(c)$ has dimension $n_{c,a}$.
Tensoring it with $\IH_c^{2j}$ gives total orbifold degree $2(a+j)$.
Since the support degree $j$ is integral, this also preserves the
residue class of the sector age modulo $\mathbb Z$, without requiring
the splitting to preserve each exact sector age separately.
Choosing graded identifications of these parts with
$\mathbb C^{n_{c,a}}$ therefore yields
\eqref{eq:weak-fano-pure-sheaf-decomposition}.

Finally, multiplication by the radial conewise-linear function
$\psi_P$ realizes cup product with $L=c_1^{\mathrm{orb}}(\mathcal
X_\Gamma)$.  Since the pure-sheaf decomposition is linear over the
conewise-polynomial coefficient sheaf, its restriction to the
$c$-supported summand is $1\otimes L_c$.  The primitive multiplicity
spaces are canonical, but the splitting above, and hence the displayed
direct-sum isomorphism, is not.
\end{proof}

Together with the primitive decomposition
\eqref{eq:quotient-fan-primitive-decomposition}, the proposition shows
that the orbifold-graded dimensions and the Lefschetz-string
multiplicities in each age class modulo $\mathbb Z$ are independent
of the chosen crepant subdivision.

\subsection{The weak Fano stringy mirror formula and its monodromy refinement}
\label{subsec:weak-fano-stringy-interface}

Let $P\subset N_{\mathbb R}$ be a positive-dimensional lattice
polytope containing $0$, put
\[
 d:=\rank N,
 \qquad r:=\dim P,
 \qquad N_P:=N\cap\Span_{\mathbb R}P,
\]
let $f\in\C[N]$ be nondegenerate at infinity with $P_{f,\infty}=P$,
and choose any quasi-projective stacky log-$\mathbb Q$-Gorenstein fan
$\bSigma$ in $N$ with convex radial support $P_{\bSigma}=P$.  Then
$(\{0\},f)$ and $(\bSigma,0)$ form a mirror pair in the ambient dual
lattices.  If $0\in\relint(P)$ and $r=d$, this is the usual weak Fano
specialization; the same formula covers every semi-projective or
non-spanning radial fan presentation.

\begin{corollary}[Stringy Fano--Landau--Ginzburg mirror formula]
\label{cor:weak-fano-harder-lee-new}
For every $f$ and every $\bSigma$ as above, one has
\begin{equation}
 \Fst(\{0\},f;X,Y)
 =X^d\Fst(\bSigma,0;X^{-1},Y).
 \label{eq:weak-fano-stringy-mirror-new}
\end{equation}
\end{corollary}

\begin{proof}
For the zero fan, the stringy invariant is the ordinary irregular
Hodge polynomial of its torus.  Thus
Corollary~\ref{cor:newton-irregular-hodge-new} gives
\begin{equation*}
 \Fst(\{0\},f;X,Y)
 =X^r(1+X)^{d-r}P_{\mathrm{orb}}(P;Y/X).
\end{equation*}
On the other hand,
Proposition~\ref{prop:stringy-orbifold-realization-new} gives
\begin{align*}
 X^d\Fst(\bSigma,0;X^{-1},Y)
 &=X^d(1+X^{-1})^{d-r}
   P_{\mathrm{orb}}(P;Y/X)\\
 &=X^r(1+X)^{d-r}
   P_{\mathrm{orb}}(P;Y/X).
\end{align*}
Comparing the two expressions proves
\eqref{eq:weak-fano-stringy-mirror-new}, the corresponding instance of
the stringy Harder--Lee formula
\cite[Corollary~7.2]{StringyICarr}.
\end{proof}

The identity above uses the specialization $v=1$ and therefore
forgets the limit-weight grading and the nilpotent part of monodromy.
Assume henceforth that $r=d$ and $0\in\operatorname{int}(P)$, choose
$\boldsymbol\Gamma\to\boldsymbol\Sigma_P$ as in the preceding
subsection, and retain the corresponding $\mathcal X_\Gamma$ and $L$.
Also retain $V_{\mathbb C}$ and $V_\lambda$ from
\eqref{eq:weak-fano-semisimple-eigenspace}.  By
Lemma~\ref{lem:laurent-middle-degree-concentration}, $V_{\mathbb C}$
is the only nonzero relative cohomology group at infinity.  The full
Newton expression determines its semisimple and nilpotent monodromy
data by Corollary~\ref{cor:newton-monodromy-multiplicities}; the
compact-side decomposition above realizes the same local
multiplicities through inertia ages and anticanonical Lefschetz
strings.

\begin{corollary}[The age--monodromy correspondence]
\label{cor:weak-fano-age-monodromy}
Let $T_\infty$ be the affine $t$-plane monodromy of
\eqref{eq:affine-and-limit-monodromy-conventions}, acting on
$V_{\mathbb C}$, and let $A_{\mathrm{age}}$ denote the semisimple
operator on $H^*_{\mathrm{orb}}(\mathcal X_\Gamma;\mathbb C)$ whose
restriction to $H^*(\mathcal X_{\mathfrak g};\mathbb C)$ is
multiplication by $e^{-2\pi i\age(\mathfrak g)}$.
\begin{enumerate}[label=\textup{(\roman*)}]
\item The semisimple part satisfies
\begin{equation}
 \det(z-T_{\infty,s})
 =\prod_{\mathfrak g\subset I\mathcal X_\Gamma}
 \left(z-e^{-2\pi i\age(\mathfrak g)}\right)^{
 \dim H^*(\mathcal X_{\mathfrak g};\mathbb C)}.
 \label{eq:weak-fano-age-characteristic-polynomial}
\end{equation}
For the convention $T_\Phi$ used by $E_\Phi^\infty$, all eigenvalues
on the right are inverted.

\item On every eigenvalue summand, $N_\infty$ and $-L$ have the same
Jordan form, where the age-$a$ sector on the compact side is assigned
the eigenvalue $e^{-2\pi ia}$.  More precisely, if $J_{k,\lambda}$
denotes the number of blocks of size $k$ on the $\lambda$-eigenspace,
then it is given by \eqref{eq:weak-fano-jordan-block-number}.  The
individual summand $n_{c,a}b_{c,j}$ is the numerical contribution of
the local term $(c,a,j)$, not a canonical label on a Jordan block of
$N_\infty$.

Consequently, over $\mathbb C$ there is a noncanonical isomorphism
\begin{equation}
 \left(V_{\mathbb C},T_{\infty,s},N_\infty\right)
 \cong_{\mathrm{noncan}}
 \left(
  H^*_{\mathrm{orb}}(\mathcal X_\Gamma;\mathbb C),
  A_{\mathrm{age}},
  -c_1^{\mathrm{orb}}(\mathcal X_\Gamma)\cup(-)
 \right).
 \label{eq:weak-fano-age-monodromy-correspondence}
\end{equation}
Equivalently, in an unnormalized logarithm convention,
\begin{equation}
 \log T_{\infty,u}
 \cong_{\mathrm{noncan}}
 -2\pi i\,c_1^{\mathrm{orb}}(\mathcal X_\Gamma)\cup(-).
 \label{eq:weak-fano-unnormalized-logarithm}
\end{equation}
\end{enumerate}
\end{corollary}

\begin{proof}
\emph{The semisimple part.}
For $V_\lambda$ as in
\eqref{eq:weak-fano-semisimple-eigenspace}, the Newton--Hodge
calculation \eqref{eq:weak-fano-eigenvalue-multiplicity-local}
gives its dimension.  On the compact side, taking dimensions in a
fixed age class $\beta\in\mathbb Q/\mathbb Z$ in
\eqref{eq:weak-fano-pure-sheaf-decomposition} gives
\begin{equation}
 \sum_{\substack{\mathfrak g\subset I\mathcal X_\Gamma\\
                  \age(\mathfrak g)\equiv\beta\; (\mathrm{mod}\,\mathbb Z)}}
 \dim H^*(\mathcal X_{\mathfrak g};\mathbb C)
 =\sum_{\substack{a\in\mathbb Q\\a\equiv\beta\; (\mathrm{mod}\,\mathbb Z)}}
 \sum_{c\in\Sigma_P}
 n_{c,a}\dim\IH_c^{2\bullet}.
 \label{eq:weak-fano-age-multiplicity-compact}
\end{equation}
Here the inertia sectors are smooth proper toric stacks, so their odd
cohomology vanishes and $\dim H^*=\dim H^{2\bullet}$.
Comparing
\eqref{eq:weak-fano-eigenvalue-multiplicity-local} and
\eqref{eq:weak-fano-age-multiplicity-compact} yields
\begin{equation}
 \dim V_\lambda
 =\sum_{\substack{\mathfrak g\subset I\mathcal X_\Gamma\\
     e^{-2\pi i\age(\mathfrak g)}=\lambda}}
 \dim H^*(\mathcal X_{\mathfrak g};\mathbb C).
 \label{eq:weak-fano-eigenspace-sector-dimension}
\end{equation}
Substitution into
\eqref{eq:weak-fano-semisimple-characteristic-by-eigenspace} proves
\eqref{eq:weak-fano-age-characteristic-polynomial}.  Replacing
$T_\infty$ by $T_\Phi=T_\infty^{-1}$ inverts every eigenvalue.

\emph{The nilpotent part.}
Proposition~\ref{prop:weak-fano-pure-sheaf-decomposition} identifies
the action of $L$ on the $(c,a)$-summand with $1\otimes L_c$, so the
primitive decomposition
\eqref{eq:quotient-fan-primitive-decomposition} gives
$n_{c,a}b_{c,j}$ blocks of size $m_c-2j+1$ on the compact side as
well.  After assigning the local degree-$a$ summand the eigenvalue
$e^{-2\pi ia}$, the number of size-$k$ blocks in its
$\lambda$-graded part is therefore the right-hand side of
\eqref{eq:weak-fano-jordan-block-number}.  By
Corollary~\ref{cor:newton-monodromy-multiplicities}, this is exactly
the number of size-$k$ blocks of $N_\infty$ on $V_\lambda$.
Thus the two nilpotent operators have the same Jordan form on every
eigenvalue summand.  The normalization
\eqref{eq:normalized-monodromy-logarithms} supplies the displayed
minus sign; multiplication by $-1$ does not change the block sizes.
Together with part~\textup{(i)}, this gives
\eqref{eq:weak-fano-age-monodromy-correspondence}; multiplying the
normalized logarithm by $2\pi i$ gives
\eqref{eq:weak-fano-unnormalized-logarithm}.
\end{proof}

In the simplicial, vertex-supported case, the resulting
eigenvalue-graded nilpotent comparison specializes to
\cite[Corollary~4.5]{WangNewtonHodge}.  The pure-sheaf formulation
above replaces the sectorwise Stanley--Reisner rings appearing there
by primitive multiplicity spaces tensored with simple pure sheaves.

The Newton formula determines the Hodge multiplicities, while the
comparison above identifies the eigenvalue-graded Jordan types over
$\mathbb C$.  The resulting isomorphism is noncanonical and does not
identify rational polarized limit mixed Hodge structures or recover
their extension and polarization data.  The sign in
\eqref{eq:weak-fano-age-monodromy-correspondence} is fixed by the
monodromy convention; a canonical operator comparison would require
additional geometric or categorical input.

\end{document}